\documentclass[11pt,a4paper]{amsart}
\usepackage{amsmath}
\usepackage{amsfonts}
\usepackage{amssymb}
\usepackage{graphicx}
\usepackage[utf8]{inputenc}
\usepackage{pdflscape}
\usepackage{float}
\usepackage{mathtools}
\usepackage{appendix}
\usepackage[colorlinks=true, linkcolor=red, citecolor=blue]{hyperref}
\usepackage{epsfig}
\usepackage{tikz}
\usepackage{mathrsfs}

\numberwithin{equation}{section}
\makeatletter
\@namedef{subjclassname@2010}{\textup{2020} Mathematics Subject Classification}
\makeatother

\newtheorem{theorem}{Theorem}[section]
\newtheorem{definition}{Definition}[section]
\newtheorem{lemma}{Lemma}[section]
\newtheorem{proposition}{Proposition}[section]
\newtheorem{corollary}{Corollary}[section]
\newtheorem{remark}{Remark}[section]
\newtheorem{claim}{\textit{Claim}}[section]

\begin{document}

\title[Rao-Nakra Sandwich beam model]{Studies on the Rao--Nakra Sandwich Beam: Well-Posedness, Dynamics, and Controllability}

	\author[Bautista]{George J. Bautista}
	\address{Instituto de Matemática e Estatística, Universidade Federal Fluminense, Campus  Gragoatá, 24210-201, Niterói (RJ), Brazil.}
		\email{\url{geojbs25@gmail.com}}
		\email{\url{jlimaco@id.uff.br}}

	\author[Capistrano--Filho]{Roberto de A. Capistrano--Filho*}
		\address{Departamento de Matem\'atica, Universidade Federal de Pernambuco, S/N Cidade Universit\'aria, 50740-545, Recife (PE), Brazil}
	\email{\url{roberto.capistranofilho@ufpe.br}}
	
	\author[Chentouf]{Boumediene Chentouf}
	\address{Department of Mathematics, College of Science, Kuwait University. Sabah AlSalem University City, P.O. Box 5969, Safat 13060, Shadadiya. Kuwait.}
			\email{\url{boumediene.chentouf@ku.edu.kw}}
			
			\author[Sierra]{Oscar Sierra Fonseca}
\address{Escola de Matemática Aplicada\\ Fundação Getúlio Vargas, Praia de Botafogo 190, 22250-900, Rio de Janeiro (RJ), Brazil}
\email{\url{oscar.fonseca@fgv.br}}

	\author[Limaco]{Juan Límaco}

%
	\thanks{$^*$Corresponding author: \url{{roberto.capistranofilho@ufpe.br}}}
	\thanks{Bautista was partially supported by FAPERJ under the program PDS-2024, grant number SEI-0260003/019497/2024. Capistrano–Filho was partially supported by the CAPES-COFECUB program grant number 88887.879175/2023-00, CNPq grants numbers 301744/2025-4 and 421573/2023-6, and Propesqi (UFPE). Límaco was partially supported by CNPq grant number 310860/2023-7.}

\numberwithin{equation}{section}

\begin{abstract}
In this work, we investigate the well-posedness, stabilization, and boundary controllability of a linear Rao–Nakra type sandwich beam. The system consists of three coupled equations that represent the longitudinal displacements of the outer layers and the transverse displacement of the composite beam, all of which are coupled with dynamical boundary conditions. In the first problem, time-varying interior damping and static boundary conditions with time-dependent delays are considered. Then, we establish the existence and uniqueness of solutions for the Cauchy problem associated with the damped system using semigroup theory and a classical result by Kato. Furthermore, employing a Lyapunov-based approach, we prove that the system's energy decays exponentially, despite the presence of time-varying weights and delays. In the second problem, we consider a boundary linear control system with dynamical boundary conditions, and prove its well-posedness. By deriving an observability inequality for the adjoint system and applying the Hilbert Uniqueness Method (HUM), we show that the system is null controllable. A key contribution of this work lies in handling the full three-equation coupled system, which involves significant difficulty due to the dynamic boundary conditions, resolved via appropriately constructed Lyapunov functionals and intermediate observability inequalities.
\end{abstract}

	\date{\today}

\maketitle

\section{Introduction}

\subsection{Model formulation and control problems}
Various types of beams are widely used in practice. Indeed, we encounter the classical Euler-Bernoulli beam, Rayleigh beam, Timoshenko, Bresse beam, sandwich beams, and micro-beam. One typical example of a sandwich beam is the so-called Rao-Nakra beam. The model describes the dynamics of a three-layer composite beam consisting of two stiff elastic face layers (or outer layers) bonded to a viscoelastic core. This configuration is commonly used in structural applications where vibration suppression and damping are desired, as the viscoelastic layer dissipates energy through shear deformation.

The model was first proposed by Rao and Nakra (1974) \cite{Rao1,Rao2} as an extension of classical beam theories, incorporating both transverse shear deformation and rotary inertia effects, which are neglected in simpler models like the Euler--Bernoulli or Timoshenko beams. The governing equations couple the longitudinal displacements of the outer layers with the transverse displacement of the beam, linked through the shear strain in the viscoelastic core. This system of partial differential equations captures the interaction between bending, stretching, and shear in sandwich structures and has become a benchmark model for studying vibration control, stabilization, and active/passive damping mechanisms in smart materials and composite structures.

In this work, we study two different boundary control problems associated with the Rao–Nakra type sandwich beam equations. The two problems share the same interior dynamics, that is, the same set of three partial differential equations (PDEs). In turn, they correspond to different control settings. The first problem addresses stabilization in the presence of time-varying interior damping and static boundary delay conditions subject to time-dependent delay effects. The second one is formulated as a boundary controllability problem, where the boundary conditions are dynamical. This approach allows us to analyze complementary aspects of the same PDE model, focusing on both stability and controllability properties in a unified framework. More precisely, we first consider the following system:
\begin{equation}\label{Rao_1_12}
\left\{
\begin{array}{ll}
\rho_1 h_1 u_{tt} - E_1 h_1 u_{xx} - k(-u + v + \alpha w_x) + a_1(t) u_t = 0, & x \in (0,L),\, t>0,\\[0.2em]
\rho_3 h_3 v_{tt} - E_3 h_3 v_{xx} + k(-u + v + \alpha w_x) + a_2(t) v_t = 0, & x \in (0,L),\, t>0,\\[0.2em]
\rho h w_{tt} + EI w_{xxxx} - \alpha k (-u + v + \alpha w_x)_x + a_3(t) w_t = 0, & x \in (0,L),\, t>0,\\[0.2em]
u(0,t) = v(0,t) = w(0,t) = w_x(0,t) = w(L,t) = 0, & t>0,\\[0.2em]
u_x(L,t) = \alpha_1 u_t(L,t) + \beta_1 u_t(L,t-\tau_1(t)), & t>0,\\
v_x(L,t) = \alpha_2 v_t(L,t) + \beta_2 v_t(L,t-\tau_2(t)), & t>0,\\
w_{xx}(L,t) = \alpha_3 w_{tx}(L,t) + \beta_3 w_{tx}(L,t-\tau_3(t)), & t>0,\\
u_t(L,t-\tau_1(0)) = f_0(t-\tau_1(0)) \in L^2(0,1), & 0 < t < \tau_1(0),\\
v_t(L,t-\tau_2(0)) = g_0(t-\tau_2(0)) \in L^2(0,1), & 0 < t < \tau_2(0),\\
w_{tx}(L,t-\tau_3(0)) = h_0(t-\tau_3(0)) \in L^2(0,1), & 0 < t < \tau_3(0),\\
(u,v,w)(x,0) = (u_0, v_0, w_0)(x), & x \in (0,L),\\
(u_t, v_t, w_t)(x,0) = (u_1, v_1, w_1)(x), & x \in (0,L),
\end{array}
\right.
\end{equation}
where $u$, $v$, and $w$ denote the longitudinal displacements of the outer layers and the transverse displacement of the composite. The constants $\rho_i, h_i, E_i, I_i$ ($i=1,2,3$) denote the density, thickness, Young's modulus, and moments of inertia of each layer, and $k$ is the shear modulus of the middle layer. Moreover, we set
$$
\rho h = \rho_1 h_1 + \rho_2 h_2 + \rho_3 h_3, \quad EI = E_1 I_1 + E_3 I_3, \quad \alpha = h_2 + \dfrac{1}{2} (h_1 + h_3).
$$

Furthermore, the boundary conditions involve time-varying delays $\tau_i(t)$, $i=1,2,3$, and we assume the existence of positive constants $\tau_{0i}$, $M_i$, and $d_i<1$ such that the delays satisfy the following assumptions:
\begin{equation}\label{eq:TauCond}
\begin{cases}
0 < \tau_{0i} \leq \tau_i(t) \leq M_i, \quad \dot{\tau}_i(t) \leq d_i < 1, & \forall t \ge 0,\\
\tau_i \in W^{2,\infty}([0,T]), & \forall T>0.
\end{cases}
\end{equation}
The damping coefficients $a_i: \mathbb{R}_+ \to (0,+\infty)$ are non-increasing $C^1$ functions and satisfy
\begin{align}\label{damping_123}
0 < a_{0i} \le a_i(t)<  b_{0i},\quad \forall t \ge 0.
\end{align}
Finally, the feedback gains $\alpha_i, \beta_i$ obey
\begin{align}
\alpha_1 &> \dfrac{|\beta_1|}{2 E_1 h_1} \left( \dfrac{(E_1 h_1)^2 + 1 - d_1}{1-d_1} \right), & 0 \le d_1 < 1, \label{eq:CCond2_1}\\
\alpha_2 &> \dfrac{|\beta_2|}{2 E_3 h_3} \left( \dfrac{(E_3 h_3)^2 + 1 - d_2}{1-d_2} \right), & 0 \le d_2 < 1, \label{eq:CCond2_2}\\
\alpha_3 &> \dfrac{|\beta_3|}{2 EI} \left( \dfrac{(EI)^2 + 1 - d_3}{1-d_3} \right), & 0 \le d_3 < 1. \label{eq:CCond2_3}
\end{align}

Now, consider the energy associated with the system \eqref{Rao_1_12}:
\begin{align}\label{eq:En}
E(t) = \dfrac{1}{2} &\left(\rho_1h_1\|u_t\|^2+E_1h_1\|u_x\|^2+\rho_3h_3\|v_t\|^2+E_3h_3\left\|v_x\right\|^2+\rho h \left\|w_t\right\|^2 +EI\left\|w_{xx}\right\|^2\right.\\
&\left.+k\left\|-u+v+\alpha w_{x}\right\|^2 \right)
+ \dfrac{\lvert \beta_1\rvert }{2}\tau_1(t) \int_0^1 u_{t}^2(L, t-\tau(t)\rho )\,d\rho \nonumber \\
& + \dfrac{\lvert \beta_2\rvert }{2}\tau_2(t) \int_0^1 v_{t}^2(L, t-\tau(t)\rho )\,d\rho + \dfrac{\lvert \beta_3\rvert }{2}\tau_3(t) \int_0^1 w_{tx}^2(L, t-\tau(t)\rho )\,d\rho. \nonumber
\end{align}
Thereafter, the natural questions appear as follows:

\vspace{0.2cm}
\noindent \textbf{Problem $P_1$:} Does the system \eqref{Rao_1_12} admit a solution?

\vspace{0.2cm}
\noindent \textbf{Problem $P_2$:} If a solution exists for \eqref{Rao_1_12}, does the time-varying delay feedback ensure exponential decay of the associated energy \eqref{eq:En}?

\vspace{0.2cm}
The second problem that we addressed in this work is a boundary control system associated with the same Rao--Nakra dynamics (PDEs). While it formally corresponds to removing the time-dependent damping terms ($a_i(\cdot) =0,$  $i=1,2,3.$) in \eqref{Rao_1_12}, it also involves a substantial modification of the boundary conditions. In particular, the boundary dynamics are reinterpreted so that they play the role of a feedback control mechanism:
\begin{equation}\label{2bbm}
\left\{
\begin{array}{ll}
\rho_1 h_1 u_{tt} - E_1 h_1 u_{xx} - k(-u + v + \alpha w_x) = 0, & x\in(0,L),\, t>0,\\[0.2em]
\rho_3 h_3 v_{tt} - E_3 h_3 v_{xx} + k(-u + v + \alpha w_x) = 0, & x\in(0,L),\, t>0,\\[0.2em]
\rho h w_{tt} + EI w_{xxxx} - \alpha k (-u + v + \alpha w_x)_x = 0, & x\in(0,L),\, t>0,\\[0.2em]
u(0,t)=v(0,t)=w_x(0,t)=w_{xxx}(0,t)=0, & t>0,\\
w_{xx}(L,t)=w_{xxx}(L,t)=0, & t>0,\\
u_{tt}(L,t) + u_x(L,t) = f_1(t), & t>0,\\
v_{tt}(L,t) + v_x(L,t) = f_2(t), & t>0,\\
w_{tt}(L,t) - u(L,t) + v(L,t) + \alpha w_x(L,t) = f_3(t), & t>0,\\
(u,v,w)(x,0) = (u_0, v_0, w_0)(x), & x \in (0,L),\\
(u_t,v_t,w_t)(x,0) = (u_1, v_1, w_1)(x), & x \in (0,L),
\end{array}
\right.
\end{equation}
where $f_i(t)$, $i=1,2,3$, are the boundary controls. Within this boundary control framework, we are naturally led to consider the following controllability problem:

\vspace{0.2cm}
\noindent \textbf{Problem $P_3$:} For the controlled system \eqref{2bbm}, given $T>0$ and initial states $(u_0,u_1,v_0,v_1,w_0,w_1)$ in a suitable function space, can one determine control inputs $f_1,f_2,f_3$ such that the solution $(u,v,w)$ satisfies
\begin{align*}
(u,v,w)(x,0) &= (u_0,v_0,w_0)(x), & x \in (0,L),\\
(u_t,v_t,w_t)(x,0) &= (u_1,v_1,w_1)(x), & x \in (0,L),
\end{align*}
and
\begin{align*}
(u,v,w)(x,T) &= (0,0,0), & x \in (0,L),\\
(u_t,v_t,w_t)(x,T) &= (0,0,0), & x \in (0,L)?
\end{align*}

If, for arbitrary $T>0$, such controls can always drive the system to the equilibrium state, then the system is said to be \emph{null controllable} or \emph{controllable to zero}.

\subsection{Background and state of the art}
The Rao-Nakra sandwich beam has been actively studied both numerically and analytically. Particularly, the stabilization and controllability problem has been the subject of mathematical endeavor. This has led to so many research papers, and hence it is quasi-impossible to review all of them. Notwithstanding, we shall restrict ourselves to the literature related to the analysis and control of solutions of the Rao-Nakra beam problem.
The first outcome goes back twenty years ago \cite{raja}, where the Riesz basis property is established for a two-layer beam. This result permits us to deduce a sufficient condition for the exponential stability as well as exact controllability of the system.  Later, a three-layer  Rao-Nakra beam is considered in \cite{raja2}, and if the control time $T$ is large enough, then the system is exactly controllable under certain other parametric restrictions.

A more general model is treated in \cite{eect} in the sense that several possible multilayers are used in the form of multilayer plates. Specifically, the model consists of $2m+1$ alternating stiff and compliant (core) layers, with stiff layers on the outside. It is assumed that the stiff layers fulfill the Kirchhoff hypothesis, while the compliant layers admit shear. The Riesz basis property is shown for the eigenfunctions of the decoupled system, and then the exponential stability is concluded for the coupled system. A similar generalized model with interior damping is considered in \cite{siam}  with all combinations of clamped and hinged boundary conditions. In turn, the control is present in either the moment or the rotation angle at one end of the beam. Then, if the damping parameters are sufficiently small, the exact controllability is proved provided that the control time $T$ is large enough. Moreover, the exact controllability holds for the undamped cases without any restriction on the parameters in the system.

When the Rao-Nakra beam is subject to magnetic effects for the piezoelectric layers,  the lack of uniform strong stabilization is demonstrated in \cite{ieee} in the case where the mechanical boundary forces for the longitudinal dynamics are removed. However, in the absence of magnetic effects, four types of feedback controllers are required in order to reach the exponential stability of the system. Such results are improved in \cite{new} for fully-dynamic and electrostatic Rao-Nakra
type models. The polynomial stability is obtained in \cite{wang} for a simplified Rao-Nakra beam system. The main result in \cite{liu} is the polynomial stability of the Rao-Nakra beam system under the action of only one viscous damping in either the beam equation or one of the wave equations. The exponential stability is obtained in \cite{pde} for the Rao-Nakra beam system via a thermoelastic type control with second sound.  The same result holds if three interior viscous damping controls act despite the presence of one time-delayed term of the same type \cite{rapo}. It turns out that this outcome remains valid if the three damping controls and the delayed term are of type Kelvin--Voigt \cite{caba}. In \cite{mon}, the long-time dynamics of a non-autonomous Rao-Nakra beam with nonlinear damping and source terms are studied. Lipschitz stability result is proved, as well as the existence of pullback attractors.

Rao-Nakra sandwich beam with fractional derivatives has also attracted some authors. In fact, three boundary dissipation of fractional derivative type are considered in \cite{villa} and polynomial stability of the system is proved. Recently, fractional interior damping terms have been incorporated in the Rao-Nakra beam system \cite{kais}. Polynomial decay of the system's energy is achieved, and then the decay rate is shown to depend on the fractional derivative parameters.
Generalized Rao-Nakra beam system is also dealt with in \cite{akil}, where the system consists of four wave equations for the longitudinal displacements and the shear angle of the top and bottom layers, and one Euler–Bernoulli beam equation for the transversal displacement. Two viscous damping controls are designed, and several stability and non-stability findings are presented depending on whether both of the top and bottom layers are directly damped or otherwise. In \cite{muki}, general stability findings are presented for the Rao-Nakra beam with memory and thermal dissipation governed by Fourier’s law of heat conduction. In \cite{spanish}, two Kelvin-Voigt-type damping controls and one memory control are proposed to show the exponential stability for the Rao-Nakra beam. In turn, it is proved in \cite{muki2} that the effect of Gurtin-Pipkin’s thermal law on the outer layers of the Rao-Nakra beam system is strong enough to ensure the exponential stability of the system. In \cite{bao}, nonlinear viscous damping controls and nonlinear source terms are introduced in the beam, and the uniform energy decay rates are proved under
certain assumptions of the parameters. Moreover, the authors prove the existence of the global attractor and generalized exponential attractors of the system.
Note that exponential as well as general energy decay rates of the energy of the Rao-Nakra beam are obtained in \cite{bao3} in the case of time-varying weights and nonlinear viscous damping controls.

When the Kelvin-Voigt damping terms act on the first and third equations of the beam, the lack of exponential stability of the system is proven, and then the polynomial stability with rate $t^{-1/4}$ is established \cite{bao4}. Note that the exponential stability is established in \cite{bao2} for the Rao-Nakra beam equation with time-varying weights and time-varying delay in the damping. This result is proved by using two approaches: the multiplier method and the observability argument. Another work deals with the generalized Rao-Nakra beam equation, and the analytic stability is obtained when all the displacements are globally damped via Kelvin-Voigt damping \cite{zai}. On the other hand, if the damping acts only on the shear angle displacements of the top and bottom layers, then the energy of the system decays polynomially with a decay rate $t^{-1/2}$.

One can also find a study on the Rao-Nakra beam system in the whole line. Indeed,  linear frictional damping controls are supposed to act on the two wave equations of the system \cite{ais1}. In this event, the author proves a polynomial decay result for the $L^2(\mathbb{R}$ norm of solutions and their higher order derivatives with respect to the space variable. It is also demonstrated that the decay rate is specified in terms of the regularity of the initial data. The asymptotic behavior of solutions in a thermoplastic Rao-Nakra beam is explored when nonlinear damping with a variable exponent occurs \cite{kho}. Exponential, polynomial, and strong stability for Rao-Nakra sandwich beams are obtained with the action of a single internal infinite memory \cite{ais2}.  Furthermore, the stability type depends on the presence of the memory term in the wave or the Bernoulli equation, and also the order of the derivative in the memory term. Unlike the works mentioned above, there are articles where it is shown that the action of localized frictional damping controls in the Rao-Nakra beam system is strong enough to force the corresponding energy to decay (see \cite{alm} for linear damping and \cite{bao5} for nonlinear damping).

Beyond the specific literature on Rao-Nakra beams, various analytical techniques have been developed for boundary and internal control of beam-type and hyperbolic systems that are relevant to the present work. For instance, the work in \cite{Akil2025} shows that passive shear damping alone is insufficient to guarantee exponential stability, and that a suitable combination of active boundary controls and distributed damping mechanisms is necessary. This observation aligns with the growing interest in boundary and hybrid control strategies for complex structures. On the other hand, nonlinear dynamic effects such as superharmonic resonance have been investigated in \cite{AlOsta2024}, emphasizing the role of refined mechanical modeling in capturing realistic behaviors of sandwich beams. From a control-theoretic perspective, recent results on systems with boundary damping and nonlocal effects \cite{Cavalcanti2025,AlMahdi2022} provide general decay frameworks that extend beyond classical exponential and polynomial stability. Moreover, models incorporating additional physical effects, such as magnetic interactions in piezoelectric beams \cite{AlMahdi2025}, demonstrate that the interplay between control design and multiphysics coupling can significantly influence the asymptotic behavior of solutions. These developments motivate the analysis of Rao--Nakra type systems under more general boundary feedback laws and time-dependent mechanisms, as considered in the present work.

Lastly, the competition between a nonlinear stabilization mechanism and a nonlinear source term for the Rao-Nakra beam is investigated in \cite{ejde}. Blow-up of solutions to the problem is established with linear and nonlinear weak damping at high initial energy. An estimate is also provided for the lower and upper bounds of the lifespan of the blow-up solution and the blow-up rate \cite{ejde}.

\subsection{Main results and paper's outline}
The first main result of this article shows that the energy $E(t)$ (see \eqref{eq:En}) decays exponentially despite the presence of a time-dependent delay. Moreover, an explicit estimate for the decay rate is provided. Of course, we begin by establishing the existence of solutions for a Cauchy problem associated with system \eqref{Rao_1_12}. This outcome is guaranteed by invoking an abstract theorem due to \cite{Kato1970}. With the existence of a solution confirmed, thus providing an answer to Problem $P_1$, we can then investigate whether the energy associated with the system decays exponentially. By applying Lyapunov’s method, it follows that the energy $E(u(t))$ decays exponentially to zero as $t \to \infty$, thereby also providing a solution to Problem $P_2$.
\vspace{0.2cm}

Before stating our first main result, we shall introduce some notations.
In the sequel, we will consider the Lebesgue space $ L^2(0, L) $, equipped with its usual inner product $ \left\langle \cdot, \cdot \right\rangle $ and norm $ \| \cdot \| $.
Next,  we introduce the following Hilbert spaces
$$
H_*^1(0, L):=\left\{u \in H^1(0, L) ; u(0)=0\right\},
$$
$$
V_*^2(0, L):=\left\{w \in H^1_0(0, L) \cap H^2(0, L) ; w_x(0)=0\right\}.
$$
Then, we define the  space $\mathcal{H}$ as
$$
\mathcal{H}=\left[H_*^1(0, L) \times L^2(0, L)\right]^2 \times V_{*}^2(0,L)\times L^2(0,L).
$$
For  $U=(u,\Psi_1, v,\Psi_2,w,\Psi_3)^{\top}, \, \, U^{\sharp}=(u^{\sharp},\Psi_1^{\sharp}, v^{\sharp},\Psi_2^{\sharp}, w^{\sharp}, \Psi_3^{\sharp})^{\top} \in \mathcal{H}$, the space $\mathcal{H}$ is  endowed with the inner product
$$
\begin{aligned}
\left\langle U, U^{\sharp}\right\rangle_\mathcal{H}= &  \rho_1h_1\left\langle \Psi_1, \Psi_1^{\sharp}\right\rangle+E_1 h_1\left\langle u_x, u^{\sharp}_x\right\rangle+\rho_3 h_3\left\langle \Psi_2, \Psi_2^{\sharp}\right\rangle+E_3 h_3\left\langle v_x, v_x^{\sharp}\right\rangle \\
& +\rho h\left\langle \Psi_3, \Psi_3^{\sharp}\right\rangle +EI \left\langle w_{xx}, w_{xx}^{\sharp}\right\rangle +k\left\langle-u+v+\alpha w_{x},-u^{\sharp}+v^{\sharp}+\alpha w_{x}^{\sharp}   \right\rangle
\end{aligned}
$$
and induced norm
\begin{align*}
\|U\|^2_\mathcal{H}=&\rho_1h_1\|\Psi_1\|^2+E_1h_1\|u_x\|^2+\rho_3h_3\|\Psi_2\|^2+E_3h_3\left\|v_x\right\|^2+\rho h \left\|\Psi_3\right\|^2 +EI\left\|w_{xx}\right\|^2\\
&+k\left\|-u+v+\alpha w_{x}\right\|^2.
\end{align*}Let the space state $$\mathcal{H}_1 = \mathcal{H} \times [L^2(0,1)]^3$$ equipped with the inner product
\begin{equation*}
\begin{aligned}
\left\langle \left(U,z_1, z_2, z_3\right), \left( U^{\sharp}, z_1^{\sharp}, z_2^{\sharp}, z_3^{\sharp}\right)
\right\rangle_{t}
=\ &
\left\langle U, U^{\sharp}
\right\rangle_{\mathcal{H}}
+ \sum_{i=1}^3\lvert \beta_i\rvert  \tau_i(t)
\left\langle
	z_i,  z_i^{\sharp}
\right\rangle_{L^2(0,1)} ,
\end{aligned}
\end{equation*}
for any $\left(U,z_1, z_2, z_3\right), \left( U^{\sharp}, z_1^{\sharp}, z_2^{\sharp}, z_3^{\sharp}\right)\in \mathcal{H}_1$.

Now, the first main result can be stated as follows:
\begin{theorem}\label{th:Lyapunov0} Suppose that the time-dependent delay functions and the time-varying weights satisfy \eqref{eq:TauCond} and \eqref{damping_123}, respectively, and the feedback gains $\alpha_i, \beta_i$ fulfill \eqref{eq:CCond2_1}-\eqref{eq:CCond2_3}, for $i=1, 2, 3$. Then, there exist three positive constants $\zeta$,  $\lambda$ and $\mu_*$ such that the energy $E(t)$ given by \eqref{eq:En}, along the solution of the system  \eqref{Rao_1_12} in $\mathcal{H}_1$ satisfies
$$
E(t) \leq \zeta e^{-\frac{\lambda}{\mu_*}t }E(0), \quad \hbox{ for all } t \geq 0.$$
\end{theorem}

On the other hand, the second main result of this paper concerns the boundary controllability of system \eqref{2bbm}. More precisely, it shows that the system can be driven from any given initial state to the equilibrium state. This provides a positive answer to Problem $P_3.$

To define the state space of \eqref{2bbm}, we introduce the following Hilbert spaces
$$
H_*^2(0, L):=\left\{w \in H^2(0, L) ; w_x(0)=0\right\}\cap \left\{w \in L^2(0, L) ; \int_{0}^L w\, dx=0\right\}
$$
and
$$
H^4_{*}(0,L):=\left\{ w\in H^4(0,L);w_{xx}(L)=w_{xxx}(L)=w_{xxx}(0)=0\right\}.
$$
Then, the phase space is defined as
$$
\mathcal{H}_2=\left[H_*^1(0, L) \times L^2(0, L)\right]^2 \times H_{*}^2(0,L)\times L^2(0,L)\times \mathbb{R}^3.
$$
For
\begin{equation}\label{U-r}
U=(u,\Psi_1, v,\Psi_2,w,\Psi_3, \Psi_4, \Psi_5, \Psi_6)^{\top}\in \mathcal{H}_2
\end{equation}
and
\begin{equation}\label{U-r-a}
U^{\sharp}=(u^{\sharp},\Psi_1^{\sharp}, v^{\sharp},\Psi_2^{\sharp}, w^{\sharp}, \Psi_3^{\sharp}, \Psi_4^{\sharp}, \Psi_5^{\sharp}, \Psi_6^{\sharp})^{\top} \in \mathcal{H}_2,
\end{equation}
the phase space is  endowed with the inner product
$$
\begin{aligned}
\left\langle U, U^{\sharp}\right\rangle_{\mathcal{H}_2}= &  \rho_1h_1\left\langle \Psi_1, \Psi_1^{\sharp}\right\rangle+E_1 h_1\left\langle u_x, u^{\sharp}_x\right\rangle+\rho_3 h_3\left\langle \Psi_2, \Psi_2^{\sharp}\right\rangle+E_3 h_3\left\langle v_x, v_x^{\sharp}\right\rangle \\
& +\rho h\left\langle \Psi_3, \Psi_3^{\sharp}\right\rangle +EI \left\langle w_{xx}, w_{xx}^{\sharp}\right\rangle +k\left\langle-u+v+\alpha w_{x},-u^{\sharp}+v^{\sharp}+\alpha w_{x}^{\sharp}   \right\rangle\\
&+E_1 h_1\left(\Psi_4,\Psi_4^{\sharp}   \right)_{\mathbb{R}}+E_3h_3\left(\Psi_5,\Psi_5^{\sharp}   \right)_{\mathbb{R}}+\alpha k\left(\Psi_6,\Psi_6^{\sharp}   \right)_{\mathbb{R}}
\end{aligned}
$$
and induced norm
\begin{align*}
\|U\|^2_{\mathcal{H}_2}=&\rho_1h_1\|\Psi_1\|^2_{L^2}+E_1h_1\|u_x\|^2_{L^2}+\rho_3h_3\|\Psi_2\|^2_{L^2}+E_3h_3\left\|v_x\right\|^2_{L^2}+\rho h \left\|\Psi_3\right\|^2_{L^2} +EI\left\|w_{xx}\right\|^2_{L^2}\\
&+k\left\|-u+v+\alpha w_{x}\right\|^2_{L^2}+E_1h_1 \left| \Psi_4 \right|^2_{\mathbb{R}}+E_3h_3\left| \Psi_5 \right|^2_{\mathbb{R}}+\alpha k\left| \Psi_6 \right|^2_{\mathbb{R}}.
\end{align*}

Our second main result is stated below:

\begin{theorem}\label{controlteorfinalnodemos}Let a time $ T>0 $ and  given  $$\left(  u_{0}%
,u_{1},v_{0},v_{1},w_{0},w_{1}, u_{1}(L), v_{1}(L), w_{1}(L)\right)  ^{\top}\in D\left(\mathcal{P}\right), $$
one can always find a control inputs  $ f_i \in  L^2(0,T), \ \ i=1,2,3  $,  such that \eqref{2bbm} admits a unique solution $$ U=(u,u_t, v,v_t,w, w_t, u_t(L), v_t(L), w_t(L))^{\top} \in
C ([0,T]; D\left(\mathcal{P}\right))$$
satisfying
$$ (u(T), u_t(T),  v(T), v_t(T),  w(T),  w_t(T), u_t(L, T), v_t(L, T), w_t(L, T))=(0, 0, 0, 0, 0, 0, 0, 0, 0).$$
Here, $\mathcal{P}$ denotes the operator associated with system \eqref{2bbm}, whose precise definition is given later in \eqref{opae2} and \eqref{DAA-1}.
\end{theorem}

To establish Theorem \ref{controlteorfinalnodemos}, we employ the observation–control duality framework developed by Dolecki and Russell \cite{Do-Ru}, within the setting introduced by Lions \cite{Lions1988} (see also \cite{Lions1988SIAM}). Specifically, we rely on the Hilbert Uniqueness Method (HUM), which provides a fundamental equivalence between controllability and observability.

A central contribution of our work is the study of a Rao--Nakra sandwich beam model involving three coupled equations. This significantly increases the complexity of the analysis, as the model consists of three beams subject to dynamic boundary conditions.  These boundary conditions present substantial challenges in proving both the exponential decay of the energy and the observability inequality. As mentioned earlier, these difficulties are overcome by developing an appropriate Lyapunov approach and by deriving two intermediate observability inequalities, which play a crucial role in handling the full system.

\vspace{0.2cm}
We conclude the introduction with an outline of the paper. In Section \ref{sec2}, we establish the well-posedness of system \eqref{Rao_1_12} using semigroup theory,
specifically applying a theorem due to Kato \cite{Kato1970}, thereby providing a solution to the first problem posed in the introduction. Section \ref{sec3} is devoted to proving our first main result, showing that the energy associated with system \eqref{Rao_1_12} decays exponentially, thus addressing the second main problem introduced in this work. Sections \ref{sec4} and \ref{sec5} focus on the properties of the control system \eqref{2bbm}. In particular, we first establish the well-posedness of this system via semigroup theory, and then derive an observability inequality for the adjoint system associated with \eqref{2bbm}, which, together with the Hilbert Uniqueness Method, ensures Theorem \ref{controlteorfinalnodemos}. Finally, the paper ends with a conclusion.

\section{Well-posedness result for the delayed system}\label{sec2}

Let us start considering the following linear abstract Cauchy problem
\begin{equation}\label{eq:Cauchy}
\begin{cases}
\dfrac{d}{dt} U(t) = A(t)U(t),& t>0,\\
 U(0) = U_0, & t>0,
\end{cases}
\end{equation}
where $A(t)\colon D(A(t))\subset {H}\to {H}$ is densely defined. If $D(A(t))$ is independent of time $t$, i.e., $D(A(t)) = D(A(0)),$ for $t > 0.$ Then, the next theorem ensures the existence and uniqueness of the Cauchy problem \eqref{eq:Cauchy}.
\begin{theorem}[\cite{Kato1970}]\label{th:KatoCauchy}
Assume that:
\begin{enumerate}
\item $\mathcal{Z} = D(A(0))$ is a dense subset of $H$ and $ D(A(t)) = D(A(0))$, for all $t > 0$.
\item 
$A(t)$ generates a strongly continuous semigroup on $H$. Moreover, the family $\{A(t) \colon t\in [0, T ]\}$ is stable with stability constants $C, \ m$ independent of $t$.
\item $\partial_t A(t)$ belongs to $L_{\ast}^\infty([0, T ], B(\mathcal{Z}, H))$, the space of equivalent classes of essentially bounded, strongly measure functions from $[0, T ]$ into the set $B(\mathcal{Z}, H)$ of bounded operators from $\mathcal{Z}$ into $H$.
\end{enumerate}
Then, problem \eqref{eq:Cauchy} has a unique solution $U \in C([0, T ], \mathcal{Z}) \cap C^1 ([0, T ], H)$ for any initial datum in $\mathcal{Z}$.
\end{theorem}

The task ahead is to apply the above result to ensure the existence of solutions for the system \eqref{Rao_1_12}. Arguing in a classical way, see for instance, \cite{xyl} and \cite{Nicaise2006, Nicaise2009}, let us define the auxiliary variables
\begin{equation}\label{z_iforchange}
\begin{cases}
z_1(\rho, t) = u_{t}(L, t- \tau_1(t)\rho), \\
z_2(\rho, t) = v_{t}(L, t- \tau_2(t)\rho), \\
z_3(\rho, t) = w_{tx}(L, t- \tau_3(t)\rho), \\
\end{cases}
\end{equation}which satisfy the following transport equations,
\begin{equation}\label{eq:tr_z1}
\begin{cases}
\tau_1(t)z_{1,t}(\rho, t) + (1-\dot\tau_1(t)\rho)z_{1, \rho}(\rho, t) = 0, & \rho \in(0,1), t>0, \\
z_1(0, t) = u_{t}(L, t), \ z_1(\rho, 0) = f_0(-\tau_1(0)\rho), &  \  \rho \in (0,1), t>0.
\end{cases}
\end{equation}
\begin{equation}\label{eq:tr_z2}
\begin{cases}
\tau_2(t)z_{2,t}(\rho, t) + (1-\dot\tau_2(t)\rho)z_{2, \rho}(\rho, t) = 0, & \rho \in(0,1), \, t>0, \\
z_2(0, t) = v_{t}(L, t), \ z_2(\rho, 0) = g_0(-\tau_2(0)\rho), &  \  \rho \in (0,1), \,  t>0,
\end{cases}
\end{equation}
\begin{equation}\label{eq:tr_z3}
\begin{cases}
\tau_3(t)z_{3,t}(\rho, t) + (1-\dot\tau_3(t)\rho)z_{3, \rho}(\rho, t) = 0, & \rho \in(0,1),  \, t>0, \\
z_3(0, t) = w_{tx}(L, t), \ z_3(\rho, 0) = h_0(-\tau_3(0)\rho) & \rho \in (0,1), \, t>0,
\end{cases}
\end{equation}
respectively.

Now, we introduce the new variables
$$
\Psi_1=u_{t},\ \ \ \Psi_2=v_{t},\quad \text{and} \quad \Psi_3=w_{t},$$
and define the vector functions
$$
V=(u,\Psi_1, v,\Psi_2,w,\Psi_3, z_1, z_2, z_3)^{\top}.$$Therefore, the systems \eqref{Rao_1_12}, \eqref{eq:tr_z1}, \eqref{eq:tr_z2} and \eqref{eq:tr_z3} can be reformulated as an abstract Cauchy problem
\begin{equation}\label{abs_11_2}
\left\{
\begin{array}
[c]{l}
V_t = \mathcal{A}(t) V\\
V(0) = V_0=\left(  u_{0} ,u_{1},v_{0},v_{1},w_{0},w_{1}, f_0(-\tau_1(0)\cdot), g_0(-\tau_2(0)\cdot), h_0(-\tau_3(0)\cdot) \right)^{\top},
\end{array}\right.
\end{equation} where  the time-dependent operator
$$\mathcal{A}(t)\colon D(\mathcal{A}(t))\subset \mathcal{H}_1\to \mathcal{H}_1$$ is given by
\begin{equation}\label{opae2_33}
\mathcal{A}(t)V:=\begin{pmatrix}
\Psi_1 \\
\dfrac{1}{\rho_1 h_1} \left[ E_1 h_1 u_{xx} + k(-u + v + \alpha w_x)-a_1(t)\Psi_1 \right] \\
\Psi_2 \\
\dfrac{1}{\rho_3 h_3} \left[ E_3 h_3 v_{xx} - k(-u + v + \alpha w_x) -a_2(t)\Psi_2 \right] \\
\Psi_3 \\
\dfrac{1}{\rho h} \left[ -EI w_{xxxx} + \alpha k (-u + v + \alpha w_x)_x -a_3(t)\Psi_3  \right] \\
\\
\dfrac{\dot\tau_1(t)\rho -1}{\tau_1(t)} z_{1,\rho} \\
\\
\dfrac{\dot\tau_2(t)\rho -1}{\tau_2(t)} z_{2,\rho} \\
\\
\dfrac{\dot\tau_3(t)\rho -1}{\tau_3(t)} z_{3,\rho}
\end{pmatrix},
\end{equation}with domain defined by
\begin{equation}\label{DAA}
D(\mathcal{A}(t))=
\left\lbrace
	\begin{aligned}
		&(u,\Psi_1, v,\Psi_2,w,\Psi_3, z_1, z_2, z_3)^{\top}\in \mathcal{H}_1; \\ & (u,v,w)\in \left[ H^2(0,L)\right]^2\times H^4(0,L),  \\
	& (\Psi_1, \Psi_2,\Psi_3)\in \left[  H_*^1(0,L) \right]^2\times V^2_{*}(0,L),(z_1, z_2, z_3) \in H^1(0,L)
    \\ &w(L)=0,\\
    &z_1(0) = \Psi_1(L),  u_{x}(L) = \alpha_1 \Psi_1(L) + \beta_1 z_1(1),\\
    &z_2(0) = \Psi_2(L),  v_{x}(L) = \alpha_2 \Psi_2(L) + \beta_2 z_2(1),\\
    &z_3(0) = \Psi_{3,x}(L),  w_{xx}(L) = \alpha_3 \Psi_{3,x}(L) + \beta_3 z_3(1)
	\end{aligned}
\right\rbrace.
\end{equation}

Now, taking the triplet $\lbrace \mathcal{A}, \mathcal{H}_1, \mathcal{Z}\rbrace$, with $\mathcal{A} = \left\lbrace \mathcal{A}(t)\colon t\in[0,T] \right\rbrace$, for some $T>0$ fixed and $\mathcal{Z} = D(\mathcal{A}(0))$, we can state and prove the well-posedness result of \eqref{eq:Cauchy} related to $\lbrace \mathcal{A}, \mathcal{H}_1, \mathcal{Z}\rbrace$. This gives us the answer for the Problem $P_1$ mentioned at the beginning of the work.

\begin{theorem} \label{well-lin}
 Assume that $\alpha_i$ and $\beta_i$ are real constants such that  \eqref{eq:CCond2_1}-\eqref{eq:CCond2_3} holds, for $i=1, 2, 3$. Taking $V_0 \in \mathcal{H}_1$,  there exists a unique solution $V \in C([0, +\infty), \mathcal{H}_1)$ to \eqref{eq:Cauchy} whose operator is defined by \eqref{opae2_33}-\eqref{DAA}. Moreover, if $V_0  \in D(\mathcal{A}(0))$, then $V \in C([0, +\infty), D(\mathcal{A}(0)))\cap C^1([0, +\infty), \mathcal{H}_1).$
\end{theorem}
\begin{proof}
The result is established in a standard manner; see, for instance, \cite{Nicaise2009}.
First, we observe that $\mathcal{Z} = D(\mathcal{A}(0))$ forms a dense subset of $\mathcal{H}_1$, and moreover, $D(\mathcal{A}(t)) = D(\mathcal{A}(0))$ for all $t>0$. Hence, condition (1) of Theorem \ref{th:KatoCauchy} is satisfied.

Regarding condition (2) of Theorem \ref{th:KatoCauchy}, we note that straightforward integrations by parts, together with the imposed boundary conditions, yield
\begin{align*}
\left\langle \mathcal{A}(t) V, V \right\rangle_t - \kappa(t) \left\langle V, V \right\rangle_t
\leq& \dfrac{1}{2} \left(\Psi_{1}( L, t), z_1(1, t) \right)
\Phi^{d_1}_{\alpha_1,\beta_1}
\left(\Psi_{1}( L, t), z_1(1, t) \right)^{{\top}}\\
&+\dfrac{1}{2} \left(\Psi_{2}( L, t), z_2(1, t) \right)
\Phi^{d_2}_{\alpha_2,\beta_2}
\left(\Psi_{2}( L, t), z_2(1, t) \right)^{{\top}} \\
&+\dfrac{1}{2} \left(\Psi_{3,x}( L, t), z_3(1, t) \right)
\Phi^{d_3}_{\alpha_3,\beta_3}
\left(\Psi_{3,x}( L, t), z_3(1, t) \right)^{{\top}},
\end{align*}where
\begin{equation}\label{ksum_forall}
\kappa(t) = \displaystyle\sum_{i=1}^3\dfrac{(\dot\tau_i(t)^2+1)^{1/2}}{2\tau_i(t)},
\end{equation}
\begin{equation}\label{matrix_nega_def_albe_111}
\Phi^{d_1}_{\alpha_1,\beta_1} = \begin{pmatrix}
-2E_1h_1\alpha_1 + \lvert\beta_1\rvert & -E_1h_1\beta_1\\
-E_1 h_1\beta_1 & \lvert\beta_1\rvert (d_1 -1 )
\end{pmatrix},
\end{equation}
\begin{equation}\label{matrix_nega_def_albe_222}
\Phi^{d_2}_{\alpha_2,\beta_2} = \begin{pmatrix}
-2E_3h_3\alpha_2 + \lvert\beta_2\rvert & -E_3h_3\beta_2\\
-E_3 h_3 \beta_2 & \lvert\beta_2\rvert (d_2 -1 )
\end{pmatrix},
\end{equation}and
\begin{equation}\label{matrix_nega_def_albe_333}
\Phi^{d_3}_{\alpha_3,\beta_3} = \begin{pmatrix}
-2E I \alpha_3 + \lvert\beta_3\rvert & -E I \beta_3\\
-E I \beta_3 & \lvert\beta_3\rvert (d_3 -1 )
\end{pmatrix}.
\end{equation}

Owing to \eqref{eq:CCond2_1}, \eqref{eq:CCond2_2}, and \eqref{eq:CCond2_3}, it follows that the matrices
$\Phi^{d_1}_{\alpha_1, \beta_1}$, $\Phi^{d_2}_{\alpha_2, \beta_2}$, and $\Phi^{d_3}_{\alpha_3, \beta_3}$
are negative definite. Consequently, we have
\[
\left\langle \mathcal{A}(t) V, V \right\rangle_t - \kappa(t) \left\langle V, V \right\rangle_t \le 0,
\]
which implies that the operator $\tilde{\mathcal{A}}(t) := \mathcal{A}(t) - \kappa(t) I$ is dissipative.

We now state the following claim:
\begin{claim}\label{CL1}
For all $t \in [0,T]$, the operator $\mathcal{A}(t)$ is maximal. Equivalently, there exists $\lambda > 0$ such that
$\lambda I - \mathcal{A}(t)$ is surjective.
\end{claim}

In fact, let us fix $t\in[0,T]$. Given $F=(f_1,f_2, f_3, f_4, f_5, f_6, h_1, h_2, h_3)^T \in \mathcal{H}_1$, we seek a solution $V=(u,\Psi_1, v,\Psi_2,w,\Psi_3, z_1, z_2, z_3)^{\top} \in D(\mathcal{A}(t))$  of the equation $(\lambda I - \mathcal{A}(t))V = F$, that is,
\begin{equation}\label{eq:WP_ewz}
\begin{cases}
\lambda u -\Psi_1 = f_1,\\
\lambda \Psi_1 -\dfrac{1}{\rho_1 h_1}[E_1 h_1 u_{xx} + k(-u+v+\alpha w_x)-a_1(t)\Psi_1] = f_2,\\
\lambda v-\Psi_2 = f_3,\\
\lambda \Psi_2 -\dfrac{1}{\rho_3 h_3}[E_3 h_3 v_{xx} - k(-u+v+\alpha w_x)-a_2(t)\Psi_2] = f_4,\\
\lambda w - \Psi_3 = f_5\\
\lambda \Psi_3 -\dfrac{1}{\rho h}[-E I w_{4x} + \alpha k(-u+v+\alpha w_x)_x-a_3(t)\Psi_3] = f_6,\\
\lambda z_i + \left(\dfrac{1-\dot\tau_i(t)\rho}{\tau_i(t)}\right) z_{i, \rho} = h_i, \, \, \, i=1, 2, 3,\\
u(0) = v(0) = w(0) = w_x(0) =w(L) =0, \\
z_1(0) = \Psi_1(L), \quad u_{x}(L) = \alpha_1 \Psi_1(L) + \beta_1 z_1(1),\\
z_2(0) = \Psi_2(L),  \quad v_{x}(L) =  \alpha_2\Psi_2(L) + \beta_2z_2(1),\\
z_3(0) = \Psi_{3,x}(L), \quad  w_{xx}(L) = \alpha_3 \Psi_{3,x}(L) + \beta_3 z_3(1).
\end{cases}
\end{equation}
One can readily verify that, for any $i =1, 2, 3$,  $z_i$ is given by
\begin{equation*}
z_i(\rho) =
\begin{cases}\displaystyle
z_{i}(0) e^{-\lambda\tau_i(t)\rho} +\displaystyle \tau_i(t) e^{-\lambda\tau_i(t)\rho} \int_0^\rho e^{\lambda\tau_i(t)\sigma} h_i (\sigma)\,d\sigma, &\!\!\!\text{if }\dot\tau_i(t) = 0, \\[1mm]
\\
\displaystyle
e^{\lambda\dfrac{\tau_i(t)}{\dot\tau_i(t)}\ln(1-\dot\tau_i(t)\rho)} \left[ z_{i}(0)
\displaystyle+
\int_0^\rho \dfrac{h_i(\sigma)\tau_i(t)}{1-\dot\tau_i(t)\sigma}  e^{-\lambda\dfrac{\tau_i(t)}{\dot\tau_i(t)}\ln(1-\dot\tau_i(t)\sigma)}\,d\sigma \right], &\!\!\! \text{if }\dot\tau_i(t) \neq 0.
\end{cases}
\end{equation*}
Thereby, $z_i(1) = z_{i}(0) g_{i,0}(t) + g_{h_i}(t)$, in which
\begin{equation*}
g_{i,0}(t) =
\begin{cases}\displaystyle
e^{-\lambda\tau_i(t)}, &\text { if } \dot\tau_i(t) = 0,\\ \displaystyle
e^{\lambda\dfrac{\tau_i(t)}{\dot\tau_i(t)}\ln(1-\dot\tau_i(t))}, & \text { if }\dot\tau_i(t) \neq 0,
\end{cases}
\end{equation*}
and
\begin{equation*}
g_{h_i}(t)=
\begin{cases}\displaystyle
 \tau_i(t) e^{-\lambda \tau_i(t)} \int_0^1 e^{\lambda \tau_i(t) \sigma} h_i(\sigma) d \sigma, & \text { if } \dot{\tau_i}(t)=0, \\
 \\
 \displaystyle
e^{\lambda \dfrac{\tau_i(t)}{\dot\tau_i(t)} \ln(1-\dot{\tau_i}(t))} \int_0^1 \dfrac{h_i(\sigma) \tau_i(t)}{1-\dot{\tau_i}(t) \sigma} e^{-\lambda \dfrac{\tau_i(t)}{\dot{\tau_i}(t)} \ln (1-\dot{\tau_i}(t) \sigma)} d \sigma, & \text { if } \dot{\tau_i}(t) \neq 0.
\end{cases}
\end{equation*}Combining the latter with \eqref{eq:WP_ewz}, it follows that $U=(u,\Psi_1, v,\Psi_2,w,\Psi_3)^{\top}$ is solution of the system
\begin{equation}\label{eq:WP_ew1}
\begin{cases}
\lambda u -\Psi_1 = f_1,\\
\lambda \Psi_1 -\dfrac{1}{\rho_1 h_1}[E_1 h_1 u_{xx} + k(-u+v+\alpha w_x)-a_1(t)\Psi_1] = f_2,\\
\lambda v-\Psi_2 = f_3,\\
\lambda \Psi_2 -\dfrac{1}{\rho_3 h_3}[E_3 h_3 v_{xx} - k(-u+v+\alpha w_x)-a_2(t)\Psi_2] = f_4,\\
\lambda w - \Psi_3 = f_5\\
\lambda \Psi_3 -\dfrac{1}{\rho h}[-E I w_{4x} + \alpha k(-u+v+\alpha w_x)_x-a_3(t)\Psi_3] = f_6,
\end{cases}
\end{equation}
and satisfy the boundary conditions
\begin{equation*}
\begin{cases}
u(0) = v(0) = w(0) = w_x(0) =w(L) = 0, \\
u_{x}(L) = (\alpha_1 +\beta_1 g_{1,0}(t))\Psi_1(L) + \beta_1 g_{h_1}(t),\\
v_{x}(L) = (\alpha_2 +\beta_2 g_{2,0}(t)) \Psi_2(L) + \beta_2 g_{h_2}(t),\\
w_{xx}(L) = (\alpha_3 +\beta_3 g_{3,0}(t))\Psi_{3,x}(L) + \beta_3 g_{h_3}(t).
\end{cases}
\end{equation*}

Now, let $ \phi_i \in C^{\infty}(\left[ 0, L \right])$, for $i=1, 2, 3$, be a functions such that $ \phi_{1}(0)=\phi_{1,x}(L)=0 $, $ \phi_{2}(0)=\phi_{2,x}(L)=0 $  and $ \phi_{3}(0)=\phi_{3,x}(0)=\phi_{3}(L)=\phi_{3,xx}(L)=\phi_{3,xxx}(L)=0 $. Next, for $i=1, 2, 3$, we define a functions $\varphi_i(x,\cdot) = \phi_i(x)\beta_i g_{h_i}(\cdot) \in C^\infty([0,L])$ and let
\begin{equation*}
\begin{cases}
\hat{u}:=u - \varphi_1,\\
\hat{v}:=v - \varphi_2,\\
\hat{w}:=w - \varphi_3.
\end{cases}
\end{equation*}This, together with \eqref{eq:WP_ew1}, implies that  $\hat{U}=(\hat{u},\Psi_1, \hat{v},\Psi_2, \hat{w}, \Psi_3)^{\top}$ satisfy
\begin{equation*}
\begin{cases}
\lambda \hat{u} -\Psi_1 = \tilde{f_1},\\
\lambda \Psi_1 -\dfrac{1}{\rho_1 h_1}[E_1 h_1 \hat{u}_{xx} + k(-\hat{u}+\hat{v}+\alpha \hat{w}_x)-a_1(t)\Psi_1] = \tilde{f_2},\\
\lambda \hat{v}-\Psi_2 = \tilde{f_3},\\
\lambda \Psi_2 -\dfrac{1}{\rho_3 h_3}[E_3 h_3 \hat{v}_{xx} - k(-\hat{u}+\hat{v}+\alpha \hat{w}_x)-a_2(t)\Psi_2] = \tilde{f_4},\\
\lambda \hat{w} - \Psi_3 = \tilde{f_5}\\
\lambda \Psi_3 -\dfrac{1}{\rho h}[-E I \hat{w}_{4x} +  \alpha k(-\hat{u}+\hat{v}+\alpha \hat{w}_x)_x-a_3(t)\Psi_3] = \tilde{f_6},
\end{cases}
\end{equation*}where
\begin{equation*}
\begin{cases}
 \tilde{f_1} = f_1-\lambda \varphi_1,\\
  \tilde{f_2} = f_2  -\dfrac{1}{\rho_1 h_1}[-E_1 h_1 \varphi_{1,xx} + k(\varphi_1-\varphi_2-\alpha \varphi_{3,x})],\\
\tilde{f_3} = f_3-\lambda \varphi_2,\\
\tilde{f_4} = f_4  -\dfrac{1}{\rho_3 h_3}[-E_3 h_3 \varphi_{2,xx} - k(\varphi_1-\varphi_2-\alpha \varphi_{3,x})],\\
\tilde{f_5} = f_5-\lambda \varphi_3\\
\tilde{f_6} = f_6  -\dfrac{1}{\rho h}[E I \varphi_{3,xxxx} + \alpha k(\varphi_1-\varphi_2-\alpha \varphi_{3,x})_x],
\end{cases}
\end{equation*}
as well as the boundary conditions
\begin{equation*}
\begin{cases}
\hat{u}(0) = \hat{v}(0) = \hat{w}(0) = \hat{w}_x(0) =\hat{w}(L)  =0, \\
\hat{u}_{x}(L) = (\alpha_1 +\beta_1 g_{1,0}(t))\Psi_1(L),\\
\hat{v}_{x}(L) = (\alpha_2 +\beta_2 g_{2,0}(t)) \Psi_2(L),\\
\hat{w}_{xx}(L) = (\alpha_3 +\beta_3 g_{3,0}(t))\Psi_{3,x}(L).
\end{cases}
\end{equation*}

For simplicity, we continue to denote the translated variables by $u$, $v$, and $w$.
It follows that $0 < g_{i,0}(t) < 1$ for all $i = 1,2,3$ (see, e.g., \cite{Boumediene2024,BaCBS2025}).
Hence, from \eqref{eq:CCond2_1}--\eqref{eq:CCond2_3}, we deduce that
\[
\tilde{\alpha}_i := \alpha_i + \beta_i g_{i,0}(t) > 0, \quad \text{for each } i = 1,2,3.
\]
Consequently, establishing Claim \ref{CL1} reduces to proving that
$\lambda I - \hat{\mathcal{A}}$ is surjective, where $\hat{\mathcal{A}}$ is defined by
\begin{equation*}
\hat{\mathcal{A}}(t)U:=\begin{pmatrix}
\Psi_1 \\
\dfrac{1}{\rho_1 h_1} \left[ E_1 h_1 u_{xx} + k(-u + v + \alpha w_x)-a_1(t)\Psi_1 \right] \\
\Psi_2 \\
\dfrac{1}{\rho_3 h_3} \left[ E_3 h_3 v_{xx} - k(-u + v + \alpha w_x) -a_2(t)\Psi_2 \right] \\
\Psi_3 \\
\dfrac{1}{\rho h} \left[ -EI w_{xxxx} + \alpha k (-u + v + \alpha w_x)_x -a_3(t)\Psi_3  \right]
\end{pmatrix},
\end{equation*}with a dense domain
\begin{equation*}
D(\hat{\mathcal{A}}(t))=
\left\lbrace
	\begin{aligned}
		&(u,\Psi_1, v,\Psi_2,w,\Psi_3)^{\top}\in \mathcal{H};\,   (u,v,w)\in \left[ H^2(0,L)\right]^2\times H^4(0,L),  \\
	& (\Psi_1, \Psi_2,\Psi_3)\in \left[  H_*^1(0,L) \right]^2\times V^2_{*}(0,L),
    \, \, w(L)=0,\\
    &u_{x}(L) = \tilde\alpha_1\Psi_1(L), \, \, v_{x}(L) = \tilde\alpha_2 \Psi_2(L),\, \, w_{xx}(L) = \tilde\alpha_3\Psi_{3,x}(L)
	\end{aligned}
\right\rbrace \subset \mathcal{H}.
\end{equation*}

Now, observe that adjoint of $\hat{\mathcal{A}}$, denoted by $\hat{\mathcal{A}}^\ast$, is given by
\begin{equation*}
\hat{\mathcal{A}}^\ast(t)U:=\begin{pmatrix}
-\Psi_1 \\
-\dfrac{1}{\rho_1 h_1} \left[ E_1 h_1 u_{xx} + k(-u + v + \alpha w_x)-a_1(t)\Psi_1 \right] \\
-\Psi_2 \\
-\dfrac{1}{\rho_3 h_3} \left[ E_3 h_3 v_{xx} - k(-u + v + \alpha w_x) -a_2(t)\Psi_2 \right] \\
-\Psi_3 \\
-\dfrac{1}{\rho h} \left[ -EI w_{xxxx} + \alpha k (-u + v + \alpha w_x)_x -a_3(t)\Psi_3  \right]
\end{pmatrix},
\end{equation*}with
\begin{equation*}
D(\hat{\mathcal{A}}^\ast(t))=
\left\lbrace
	\begin{aligned}
		&(u,\Psi_1, v,\Psi_2,w,\Psi_3)^{\top}\in \mathcal{H};\,   (u,v,w)\in \left[ H^2(0,L)\right]^2\times H^4(0,L),  \\
	& (\Psi_1, \Psi_2,\Psi_3)\in \left[  H_*^1(0,L) \right]^2\times V^2_{*}(0,L),
    \, \, w(L)=0,\\
    &u_{x}(L) = -\tilde\alpha_1\Psi_1(L), \, \, v_{x}(L) = -\tilde\alpha_2 \Psi_2(L),\, \, w_{xx}(L) = -\tilde\alpha_3\Psi_{3,x}(L)
	\end{aligned}
\right\rbrace.
\end{equation*}
One can readily verify that
\begin{align*}
\left\langle  \hat{\mathcal{A}} U,U
\right\rangle_{\mathcal{H}} =&-a_1(t)\Vert \Psi_1\Vert^2 -a_2(t)\Vert \Psi_2\Vert^2 -a_3(t)\Vert \Psi_3\Vert^2 \\
&- E_1 h_1\tilde\alpha_1 \Psi_1^2(L)- E_3 h_3\tilde\alpha_2 \Psi_2^2(L)- E I\tilde\alpha_3 \Psi_{3,x}^2(L),
\end{align*}and
\begin{align*}
\left\langle  \hat{\mathcal{A}}^{\ast} U,U
\right\rangle_{\mathcal{H}} =&-a_1(t)\Vert \Psi_1\Vert^2 -a_2(t)\Vert \Psi_2\Vert^2 -a_3(t)\Vert \Psi_3\Vert^2 \\
&- E_1 h_1\tilde\alpha_1 \Psi_1^2(L)- E_3 h_3\tilde\alpha_2 \Psi_2^2(L)- E I\tilde\alpha_3 \Psi_{3,x}^2(L).
\end{align*}
This guarantees that the operators $\hat{\mathcal{A}}$ and $\hat{\mathcal{A}}^\ast$ are dissipative and hence, the desired result follows from the Lumer-Phillips Theorem (see, e.g., \cite{Pazy}), which establishes Claim \ref{CL1}.

Consequently, $\tilde{\mathcal{A}}(t)$ generates a strongly continuous semigroup on $\mathcal{H}$, and the family
$\tilde{\mathcal{A}} = \{\tilde{\mathcal{A}}(t) : t \in [0,T]\}$ forms a stable family of generators in $\mathcal{H}$
with a stability constant independent of $t$. Therefore, condition (2) of Theorem \ref{th:KatoCauchy} is satisfied.

Lastly, since $\tau_i \in W^{2,\infty}([0, T ])$ for all $T>0$, and $i=1, 2, 3$, we reach that
$$
\dot{\kappa}(t)=\sum_{i=1}^3\left[\dfrac{\ddot{\tau_i}(t) \dot{\tau_i}(t)}{2 \tau_i(t)\left(\dot{\tau_i}(t)^2+1\right)^{1 / 2}}-\dfrac{\dot{\tau_i}(t)\left(\dot{\tau_i}(t)^2+1\right)^{1 / 2}}{2 \tau_i(t)^2}\right]
$$
is bounded on $[0, T]$, for all $T>0$, and
\begin{equation*}
\dfrac{d}{dt}\mathcal{A}(t)V=\begin{pmatrix}
0 \\
-\dot a_1(t)\Psi_1  \\
0 \\
-\dot a_2(t)\Psi_2  \\
0 \\
-\dot a_3(t)\Psi_3   \\
\\
\dfrac{\ddot{\tau_1}(t) \tau_1(t) \rho-\dot{\tau_1}(t)(\dot{\tau_1}(t) \rho-1)}{\tau_1(t)^2} z_{1,\rho} \\
\\
\dfrac{\ddot{\tau_2}(t) \tau_2(t) \rho-\dot{\tau_2}(t)(\dot{\tau_2}(t) \rho-1)}{\tau_2(t)^2} z_{2,\rho} \\
\\
\dfrac{\ddot{\tau_3}(t) \tau_3(t) \rho-\dot{\tau_3}(t)(\dot{\tau_3}(t) \rho-1)}{\tau_3(t)^2} z_{3,\rho}
\end{pmatrix}.
\end{equation*}
Moreover, the coefficients of $z_{i,\rho}$ and $\Psi_i$ are bounded on $[0, T]$ for $i = 1, 2, 3$,
so that the regularity condition (3) of Theorem \ref{th:KatoCauchy} is satisfied.

In summary, we have verified all the assumptions of Theorem \ref{th:KatoCauchy}.
Hence, for each $V_0 \in D({\mathcal{A}}(0))$, the Cauchy problem
\begin{equation*}
\begin{cases}
\tilde{V}_t(t) = \tilde{\mathcal{A}}(t) \tilde{V}(t), \\
\tilde{V}(0) = V_0, \quad t>0,
\end{cases}
\end{equation*}
admits a unique solution
\[
\tilde{V} \in C([0, \infty), \mathcal{H}_1) \cap C([0, \infty), D({\mathcal{A}}(0))) \cap C^1([0, \infty), \mathcal{H}_1).
\]
Therefore, the solution of \eqref{eq:Cauchy} is explicitly given by
\[
V(t) = e^{\int_0^t \kappa(s)\, ds} \, \tilde{V}(t),
\]
which proves Theorem \ref{well-lin}.
\end{proof}

We also observe that the energy $E(t)$ associated with system \eqref{Rao_1_12} satisfies the following relation.
\begin{proposition}\label{pr:Diss}
 Suppose $\alpha_i$ and $\beta_i$ are real constants such \eqref{eq:CCond2_1}-\eqref{eq:CCond2_3} holds, for $i=1, 2, 3$.  Then, for any mild solution of \eqref{abs_11_2}, the energy $E(t)$ defined by \eqref{eq:En} is non-increasing and
\begin{align}\label{eq:EnDiss3}
\dfrac{d}{dt}E(t)
 =&-a_1(t)\Vert u_t\Vert^2 -a_2(t)\Vert v_t\Vert^2 -a_3(t)\Vert w_t\Vert^2 \\
 &+
 \dfrac{1}{2} \begin{pmatrix}
u_t( L, t) \\ u_{t}(L, t-\tau_1(t))
\end{pmatrix}^{{\top}}
\Phi^{\dot\tau_1(t)}_{\alpha_1,\beta_1}
\begin{pmatrix}
u_{t}(L, t) \\ u_{t}(L, t-\tau_1(t))
\end{pmatrix}\nonumber \\
 &+
 \dfrac{1}{2} \begin{pmatrix}
v_t( L, t) \\ v_{t}(L, t-\tau_2(t))
\end{pmatrix}^{{\top}}
\Phi^{\dot\tau_2(t)}_{\alpha_2,\beta_2}
\begin{pmatrix}
v_{t}(L, t) \\ v_{t}(L, t-\tau_2(t))
\end{pmatrix}\nonumber \\
 &+
 \dfrac{1}{2} \begin{pmatrix}
w_{tx}( L, t) \\ w_{tx}(L, t-\tau_3(t))
\end{pmatrix}^{{\top}}
\Phi^{\dot\tau_3(t)}_{\alpha_3,\beta_3}
\begin{pmatrix}
w_{tx}(L, t) \\ w_{tx}(L, t-\tau_3(t))
\end{pmatrix},\nonumber
\end{align}
 where the matrices $\Phi^{\dot\tau_i(t)}_{\alpha_i,\beta_i}$ are given by \eqref{matrix_nega_def_albe_111}, \eqref{matrix_nega_def_albe_222}  and \eqref{matrix_nega_def_albe_333}, respectively,  with ${\dot\tau_i(t)}$ instead of $d_i$, for $i= 1, 2, 3$.
\end{proposition}
\begin{proof}
Considering the energy identity \eqref{eq:En}, we define
\begin{align*}
E_1(t) = \dfrac{1}{2} &\left(\rho_1h_1\|u_t\|^2+E_1h_1\|u_x\|^2+\rho_3h_3\|v_t\|^2+E_3h_3\left\|v_x\right\|^2+\rho h \left\|w_t\right\|^2 +EI\left\|w_{xx}\right\|^2\right.\\
&\left.+k\left\|-u+v+\alpha w_{x}\right\|^2 \right).
\end{align*}
We multiply equations $\eqref{Rao_1_12}_1$, $\eqref{Rao_1_12}_2$, $\eqref{Rao_1_12}_3$ by $u_t$, $v_t$, $w_t$, respectively. Then, by integrating by parts on $(0,L)$ and using the boundary conditions, we obtain
\begin{equation}\label{aux_1}
\begin{split}
\dfrac{d}{dt}E_1(t)
 =&-a_1(t)\Vert u_t\Vert^2 -a_2(t)\Vert v_t\Vert^2 -a_3(t)\Vert w_t\Vert^2 \\
 &-E_1h_1\alpha_1u^2_t(L,t)-E_1h_1\beta_1u_t(L,t) u_t(L,t-\tau_1(t)) \\
 &-E_3h_3\alpha_2v^2_t(L,t)-E_3h_3\beta_2  v_t(L,t) v_t(L,t-\tau_2(t))\\
 &-E I\alpha_3w^2_{tx}(L,t)-E I\beta_3w_{tx}(L,t) w_{tx}(L,t-\tau_3(t)).
 \end{split}
\end{equation}
On the other hand, from  \eqref{z_iforchange}-\eqref{eq:tr_z3}, for $i=1, 2, 3$, we deduce that
\begin{align}\label{aux_2}
\frac{d}{dt}\left(\dfrac{\lvert \beta_2\rvert }{2}\tau_i(t) \int_0^1 z_i^2(L, t-\tau_i(t)\rho )\,d\rho \right)=\frac{|\beta_i|}{2}(\dot\tau_i(t)-1)z_i^2(1, t)  + \dfrac{\lvert\beta_i\rvert}{2}z_i^2(0,t).
\end{align} Thus, from \eqref{eq:En}, \eqref{aux_1}, \eqref{aux_2} and using the auxiliary variables  \eqref{z_iforchange}, it results in
\begin{equation*}
\begin{split}
\dfrac{d}{dt}E(t)
 =&-a_1(t)\Vert u_t\Vert^2 -a_2(t)\Vert v_t\Vert^2 -a_3(t)\Vert w_t\Vert^2 \\
 &\frac{1}{2}\left((-2E_1h_1\alpha_1+|\beta_1|)u^2_t(L,t)-2E_1h_1\beta_1u_t(L,t) u_t(L,t-\tau_1(t))\right.\\
 &\left.+|\beta_1|(\dot\tau_1(t)-1)u^2_t(L,t-\tau_1(t))\right)\\
 &\frac{1}{2}\left((-2E_3h_3\alpha_2+|\beta_2|)v^2_t(L,t)-2E_3h_3\beta_2  v_t(L,t) v_t(L,t-\tau_2(t))\right.\\
 &\left.+|\beta_2|(\dot\tau_2(t)-1)v^2_t(L,t-\tau_1(t))\right)\\
 &\frac{1}{2}\left((-2E I\alpha_3+|\beta_3)|w^2_{tx}(L,t)-2E I\beta_3w_{tx}(L,t) w_{tx}(L,t-\tau_3(t))\right.\\
 &\left.+|\beta_3|(\dot\tau_3(t)-1)w^2_{tx}(L,t-\tau_1(t))\right)\\
 =&-a_1(t)\Vert u_t\Vert^2 -a_2(t)\Vert v_t\Vert^2 -a_3(t)\Vert w_t\Vert^2 \\
 &+
 \dfrac{1}{2} \begin{pmatrix}
u_t( L, t) \\ u_{t}(L, t-\tau_1(t))
\end{pmatrix}^{\top}
\Phi^{\dot\tau_1(t)}_{\alpha_1,\beta_1}
\begin{pmatrix}
u_{t}(L, t) \\ u_{t}(L, t-\tau_1(t))
\end{pmatrix}\nonumber \\
 &+
 \dfrac{1}{2} \begin{pmatrix}
v_t( L, t) \\ v_{t}(L, t-\tau_2(t))
\end{pmatrix}^{\top}
\Phi^{\dot\tau_2(t)}_{\alpha_2,\beta_2}
\begin{pmatrix}
v_{t}(L, t) \\ v_{t}(L, t-\tau_2(t))
\end{pmatrix}\nonumber \\
 &+
 \dfrac{1}{2} \begin{pmatrix}
w_{tx}( L, t) \\ w_{tx}(L, t-\tau_3(t))
\end{pmatrix}^{\top}
\Phi^{\dot\tau_3(t)}_{\alpha_3,\beta_3}
\begin{pmatrix}
w_{tx}(L, t) \\ w_{tx}(L, t-\tau_3(t))
\end{pmatrix},\nonumber
 \end{split}
\end{equation*} obtaining \eqref{eq:EnDiss3}.
\end{proof}

Now, we are ready to state the following result:
\begin{proposition}\label{pr:Kato}
Assume that $\alpha_i$ and $\beta_i$ are real constants such that  \eqref{eq:CCond2_1}-\eqref{eq:CCond2_3} holds, for $i=1, 2, 3$. Then, for $V= (U, z_1, z_2, z_3)^{\top}=(u,u_t, v, v_t, w, w_t, z_1, z_2, z_3)^{\top}$, solution of \eqref{abs_11_2}, the following estimate holds:
	\begin{equation}\label{eq:Kato1}
	\begin{split}
		\lVert
		U(t)
		\rVert_{\mathcal{H}}^2
		&+ \sum_{i=1}^3\lvert \beta_i \rvert \tau_{0i}
		\lVert
		z_i(\cdot, t)
		\rVert_{L^2(0,1)}^2\leq
		\lVert
		U_0
		\rVert_{\mathcal{H}}^2
		+ \lvert \beta_1 \rvert\tau_1(0)
		\lVert
		f_0(-\tau_1(0)\cdot)
		\rVert_{L^2(0,1)}^2 \\
		&+ \lvert \beta_2 \rvert\tau_2(0)
		\lVert
		g_0(-\tau_2(0)\cdot)
		\rVert_{L^2(0,1)}^2 +\lvert \beta_3 \rvert\tau_3(0)
		\lVert
		h_0(-\tau_3(0)\cdot)
		\rVert_{L^2(0,1)}^2 .
	\end{split}
	\end{equation}
	Furthermore, for every initial condition $V_0\in \mathcal{H}_1$, we have that
	\begin{equation}\label{eq:Katotr0}
	\begin{split}
		&\lVert u_{t}(L, \cdot) \rVert_{L^2(0,T)}^2+\lVert v_{t}(L, \cdot) \rVert_{L^2(0,T)}^2+\lVert w_{tx}(L, \cdot) \rVert_{L^2(0,T)}^2
		+\sum_{i=1}^3\lVert z_i(1, \cdot) \rVert_{L^2(0,T)}^2\\
		\leq&
		\lVert
		U_0
		\rVert_{\mathcal{H}}^2
		+ \lvert \beta_1 \rvert\tau_1(0)
		\lVert
		f_0(-\tau_1(0)\cdot)
		\rVert_{L^2(0,1)}^2+ \lvert \beta_2 \rvert\tau_2(0)
		\lVert
		g_0(-\tau_2(0)\cdot)
		\rVert_{L^2(0,1)}^2 \\&+\lvert \beta_3 \rvert\tau_3(0)
		\lVert
		h_0(-\tau_3(0)\cdot)
		\rVert_{L^2(0,1)}^2.
		\end{split}
	\end{equation}
On the other hand, for the initial datum, we have the following estimates
\begin{equation}\label{eq:Kato3}
\begin{split}
\lVert U_0\rVert_{\mathcal{H}}^2 \leq& \dfrac{1}{T}\lVert U \rVert_{L^2(0,T; \mathcal{H})}^2+2\sum_{i=1}^3 \Vert a_i\Vert_{L^{\infty}(0, T)}\Vert (u_t, v_t, w_t) \Vert_{[L^2(0,T; L^2(0, L))]^3}^2
\\&+E_1h_1(2\alpha_1+\lvert\beta_1\rvert)\lVert u_{t}(L, \cdot)\rVert_{L^2(0,T)}^2 + E_3 h_3(2\alpha_2+\lvert\beta_2\rvert)\lVert v_{t}(L, \cdot)\rVert_{L^2(0,T)}^2 \\
&  +E I(2\alpha_3+\lvert\beta_3\rvert)\lVert w_{tx}(L, \cdot)\rVert_{L^2(0,T)}^2 +\sum_{i=1}^3\lvert\beta_i\rvert
		\lVert
		z_i(1, \cdot)
		\rVert_{L^2(0,T)}^2
		\end{split}
	\end{equation}
	and
	\begin{equation}\label{eq:Kato4}
    \begin{cases}
\lVert f_0(-\tau_1(0)\cdot)\rVert_{L^2(0,1)}^2 \leq
		C_1(d_1,M_1)\left(\lVert z_1(\cdot, T )\rVert_{L^2(0,1)}
		+  \lVert z_1(1,\cdot) \rVert_{L^2(0,T)}^2\right),\\
    \lVert g_0(-\tau(0)\cdot)\rVert_{L^2(0,1)}^2 \leq
		C_2(d_2,M_2)\left(\lVert z_2(\cdot, T )\rVert_{L^2(0,1)}
		+  \lVert z_2(1, \cdot) \rVert_{L^2(0,T)}^2\right),\\
        \lVert h_0(-\tau(0)\cdot)\rVert_{L^2(0,1)}^2 \leq
		C_3(d_3,M_3)\left(\lVert z_3(\cdot, T )\rVert_{L^2(0,1)}
		+  \lVert z_3(1, \cdot) \rVert_{L^2(0,T)}^2\right).
    \end{cases}
	\end{equation}
\end{proposition}

\begin{proof}
Using \eqref{eq:EnDiss3} and the fact that $\Phi^{\dot\tau_i(t)}_{\alpha_i,\beta_i}$ are symmetric negative definite matrices, we deduce the existence of a positive constant $C$, such that
\begin{align*}
E^{\prime}(t)
 &+a_1(t)\Vert u_t\Vert^2 +a_2(t)\Vert v_t\Vert^2 +a_3(t)\Vert w_t\Vert^2 =
 \dfrac{1}{2} \begin{pmatrix}
u_t( L, t) \\z_1(1, t)
\end{pmatrix}^{{\top}}
\Phi^{\dot\tau_1(t)}_{\alpha_1,\beta_1}
\begin{pmatrix}
u_{t}(L, t) \\ z_1(1, t)
\end{pmatrix}\nonumber \\
 &+
 \dfrac{1}{2} \begin{pmatrix}
v_t( L, t) \\ z_2(1, t)
\end{pmatrix}^{{\top}}
\Phi^{\dot\tau_2(t)}_{\alpha_2,\beta_2}
\begin{pmatrix}
v_{t}(L, t) \\ z_2(1, t)
\end{pmatrix} +
 \dfrac{1}{2} \begin{pmatrix}
w_{tx}( L, t) \\ z_3(1, t)
\end{pmatrix}^{{\top}}
\Phi^{\dot\tau_3(t)}_{\alpha_3,\beta_3}
\begin{pmatrix}
w_{tx}(L, t) \\ z_3(1, t)
\end{pmatrix} \nonumber\\
& \leq -C \left(u^2_t( L, t)+v^2_t( L, t)+w^2_{tx}( L, t) +\sum_{i=1}^3z_i^2(1, t)\right).
\end{align*}Thus, it follows from the above estimate that
 \begin{align} \label{se_limi_ener_2}
   E^{\prime}(t) +u^2_t( L, t)+v^2_t( L, t)+w^2_{tx}( L, t) +\sum_{i=1}^3z_i^2(1, t) \leq 0
 \end{align} Integrating \eqref{se_limi_ener_2} in $[0,s]$, for $0\leq s \leq T$, we get
	\begin{equation*}
		E(s) + \int_0^s \left(u^2_t( L, t)+v^2_t( L, t)+w^2_{tx}( L, t)\right)\,dt + \sum_{i=1}^3\int_0^s z_i^2(1, t)\,dt \leq E(0),
	\end{equation*}
	and \eqref{eq:Kato1} is obtained. Taking $s=T$ and since $E(t)$ is a non-increasing function, thanks to the Proposition \ref{pr:Diss}, estimate \eqref{eq:Katotr0} holds.

    Now,  we multiply the first equation of the system \eqref{Rao_1_12}  by $(T-t)u_t$, the second one by $(T-t)v_t$, third one by $(T-t)w_t$ and integrates by parts on $(0, L)\times (0, T)$ we obtain
\begin{equation}\label{de_for_init_v}
\begin{split}
\dfrac{T}{2}\Vert U(0) \Vert^2_{\mathcal{H}} =& \dfrac{1}{2}\int^T_0\Vert U(t) \Vert^2_{\mathcal{H}}\, dt + \int^T_0(T-t)\left(a_1(t)\Vert u_t\Vert^2 +a_2(t)\Vert v_t\Vert^2 +a_3(t)\Vert w_t\Vert^2 \right) dt \\
& -E_1h_1 \int^T_0(T-t)u_x(L,t)u_t(L, t)dt -E_3h_3 \int^T_0(T-t)v_x(L,t)v_t(L, t)dt\\
&-E I \int^T_0(T-t)w_{xx}(L,t)w_{tx}(L, t)dt.
\end{split}
\end{equation}
Since
\begin{align*}
-E_1h_1 \int^T_0(T-t)u_x(L,t)u_t(L, t)dt \leq & TE_1h_1\left( \alpha_1  + \dfrac{|\beta_1|}{2}\right)\lVert u_t(L, \cdot) \rVert_{L^2(0,T)}^2 \\
&+ \dfrac{T}{2}|\beta_1|\lVert z_1(1, \cdot) \rVert_{L^2(0,T)}^2,
\end{align*}
\begin{align*}
-E_3h_3 \int^T_0(T-t)v_x(L,t)v_t(L, t)dt \leq & TE_3h_3\left( \alpha_2  + \dfrac{|\beta_2|}{2}\right)\lVert v_t(L, \cdot) \rVert_{L^2(0,T)}^2 \\
&+ \dfrac{T}{2}|\beta_2|\lVert z_2(1, \cdot) \rVert_{L^2(0,T)}^2,
\end{align*}and
\begin{align*}
 -E I \int^T_0(T-t)w_{xx}(L,t)w_{tx}(L, t)dt    \leq & TE I\left( \alpha_3  + \dfrac{|\beta_3|}{2}\right)\lVert w_{tx}(L, \cdot) \rVert_{L^2(0,T)}^2 \\
&+ \dfrac{T}{2}|\beta_3|\lVert z_3(1, \cdot) \rVert_{L^2(0,T)}^2,
\end{align*}
thanks to the equality \eqref{de_for_init_v} and inequalities above, we have that \eqref{eq:Kato3}. The proof of estimate \eqref{eq:Kato4} follows analogously to that in \cite{Boumediene2024}, and the details are therefore omitted. Hence, the proof of Proposition \ref{pr:Kato} is complete.
\end{proof}

\section{Stabilization result}\label{sec3}
Let us prove the first main result of our work, that is, the system \eqref{Rao_1_12} decays exponentially, giving the answers to the Problem $P_2$.
\begin{proof}[Proof of Theorem \ref{th:Lyapunov0}] Consider the following Lyapunov functional $$\mathcal{L}(t) = \mu E(t) +\sum_{i=0}^3\mu_i \mathcal{L}_i(t),$$ where $\mu_i \in \mathbb{R}^+$ will be chosen later. Here, $E(t)$ is the total energy given by \eqref{eq:En},   while
$$
\mathcal{L}_0(t) = \rho_1 h_1\int_0^L u u_t\,dx + \rho_3 h_3\int_0^L v v_t\,dx + \rho h\int_0^L w w_t\,dx $$
and $$ \mathcal{L}_i(t) = \dfrac{\lvert\beta_i\rvert}{2} \tau_i(t) \int_0^1 (1-\rho) z_i^2(\rho, t)\,d\rho, \,\,  i=1, 2, 3.$$

\noindent Observe that using Poincar\'e inequality, we get
\begin{align*}
| \mathcal{L}_0(t)|   \leq &\dfrac{1}{2}\left(\rho_1h_1\|u_t\|^2+\rho_1 h_1\|u\|^2+\rho_3h_3\|v_t\|^2+\rho_3h_3\left\|v\right\|^2+\rho h \left\|w_t\right\|^2 +\rho h\left\|w\right\|^2\right)\\
\leq &\dfrac{1}{2}\left(\rho_1h_1\|u_t\|^2+\dfrac{\rho_1L^2}{E_1\pi^2}E_1h_1\|u_x\|^2+\rho_3h_3\|v_t\|^2+\dfrac{\rho_3L^2}{E_3\pi^2}E_3h_3\left\|v_x\right\|^2+\rho h \left\|w_t\right\|^2 \right.\\
&\left.  +\dfrac{\rho h L^4}{E I \pi^4}EI\left\|w_{xx}\right\|^2 +k\left\|-u+v+\alpha w_{x}\right\|^2 \right).
\end{align*}Since
\begin{align*}
|\mathcal{L}_i(t)| \leq \dfrac{\lvert\beta_i\rvert}{2} \tau_i(t) \Vert z_i^2(\cdot, t)\Vert^2,
\end{align*}we obtain that
\begin{align}\label{equinerlia}
(\mu-\mu_4)E(t) \leq \mathcal{L}(t)\leq (\mu+\mu_4)E(t).
\end{align}
Here,
\begin{equation}\label{mu4}
\mu_4 = \max \left\lbrace  \mu_0, \dfrac{\mu_0\rho_1 L^2}{E_1 \pi^2}, \dfrac{\mu_0\rho_3 L^2}{E_3 \pi^2}, \dfrac{\mu_0\rho_1 L^2}{E_1 \pi^2}, \dfrac{\mu_0\rho h L^4}{E I  \pi^4}, \mu_1, \mu_2, \mu_3 \right\rbrace,
\end{equation}
and, initially,  we will take $\mu_i>0$, $i=0, 1, 2, 3$, small enough such that $\mu_4<\mu$.

On the other hand, using the system \eqref{Rao_1_12} and the boundary condition, we get that
\begin{equation}\label{de_V_1_im}
\begin{split}
\mathcal{L}^{\prime}_0(t)=&\rho_1 h_1\int_0^L u u_{tt}\,dx + \rho_3 h_3\int_0^L v v_{tt}\,dx + \rho h\int_0^L w w_{tt}\,dx \\
&+ \rho_1 h_1\int_0^L  u^2_t\,dx + \rho_3 h_3\int_0^L  v^2_t\,dx + \rho h\int_0^L  w^2_t\,dx \\
=& \rho_1h_1\|u_t\|^2-E_1h_1\|u_x\|^2+\rho_3h_3\|v_t\|^2-E_3h_3\left\|v_x\right\|^2+\rho h \left\|w_t\right\|^2 \\
&   -EI\left\|w_{xx}\right\|^2-k\left\|-u+v+\alpha w_{x}\right\|^2 + E_1h_1u_x(L,t)u(L,t) \\
& +E_3h_3v_x(L,t)v(L,t)+E I w_{xx}(L,t)w_x(L,t)\\
& - a_1(t)\int_0^L u u_{t}\,dx - a_2(t)\int_0^L v v_{t}\,dx - a_3(t)\int_0^L w w_{t}\,dx.
\end{split}
\end{equation}
Applying Young's inequality, the Sobolev embedding, and using the boundary conditions of the system \eqref{Rao_1_12}, the following estimate holds
\begin{equation}\label{1_boun_u}
\begin{split}
\left\vert E_1h_1u_x(L,t)u(L,t) \right\vert &\leq E_1h_1 \epsilon_1\Vert u_x \Vert^2 + C_{\epsilon_1, E_1, h_1}|u_x(L,t)|^2 \\
& \leq E_1h_1 \epsilon_1\Vert u_x \Vert^2 + C_{\epsilon_1, E_1, h_1}\begin{pmatrix}
		u_{t}(L,t) \\ z_{1}(1, t)
	\end{pmatrix}^{{\top}}
	\begin{pmatrix}
		\alpha_1^2 & \alpha_1\beta_1 \\
		\alpha_1\beta_1 & \beta_1^2
	\end{pmatrix}
	\begin{pmatrix}
		u_{t}(L,t) \\ z_{1}(1, t)
	\end{pmatrix},
\end{split}
\end{equation}
for any $\epsilon_1>0$ and some positive constant $C_{\epsilon_1, E_1, h_1}$.

\noindent Similarly, for any $\epsilon_2, \epsilon_3 >0$, we deduce that
\begin{align}\label{2_boun_v}
\left\vert E_3h_3 v_x(L,t)v(L,t) \right\vert  \leq E_3h_3 \epsilon_2\Vert v_x \Vert^2 + C_{\epsilon_2, E_3, h_3}\begin{pmatrix}
		v_{t}(L,t) \\ z_{2}(1, t)
	\end{pmatrix}^{{\top}}
	\begin{pmatrix}
		\alpha_2^2 & \alpha_2\beta_2 \\
		\alpha_2\beta_2 & \beta_2^2
	\end{pmatrix}
	\begin{pmatrix}
		v_{t}(L,t) \\ z_{2}(1, t)
	\end{pmatrix},
\end{align}and
\begin{align}\label{3_boun_w}
\left\vert E I w_{xx}(L,t)w_x(L,t) \right\vert \leq E I \epsilon_3\Vert w_{xx} \Vert^2 + C_{\epsilon_3, E, I} \begin{pmatrix}
		w_{tx}(L,t) \\ z_{3}(1, t)
	\end{pmatrix}^{{\top}}
	\begin{pmatrix}
		\alpha_3^2 & \alpha_3\beta_3 \\
		\alpha_3\beta_3 & \beta_3^2
	\end{pmatrix}
	\begin{pmatrix}
		w_{tx}(L,t) \\ z_{3}(1, t)
	\end{pmatrix}.
\end{align}

Additionally to that, thanks \eqref{eq:tr_z1}-\eqref{eq:tr_z3}, and using  integration by parts, we obtain that
\begin{equation}\label{deri_L1}
	\mathcal{L}_1^{\prime}(t)  = -\dfrac{\lvert\beta_1\rvert}{2}\int_0^1(1-\dot\tau_1(t)\rho)z_1^2(\rho, t) d\rho + \dfrac{\lvert\beta_1\rvert}{2}u_t^2(L,t),
\end{equation}
\begin{equation}\label{deri_L2}
	\mathcal{L}_2^{\prime}(t)  = -\dfrac{\lvert\beta_2\rvert}{2}\int_0^1(1-\dot\tau_2(t)\rho)z_2^2(\rho, t) d\rho + \dfrac{\lvert\beta_1\rvert}{2}v_t^2(L,t),
\end{equation}and
\begin{equation}\label{deri_L3}
	\mathcal{L}_3^{\prime}(t)  = -\dfrac{\lvert\beta_3\rvert}{2}\int_0^1(1-\dot\tau_3(t)\rho)z_3^2(\rho, t) d\rho + \dfrac{\lvert\beta_3\rvert}{2}w_{tx}^2(L,t).
\end{equation}
Moreover, using \eqref{damping_123}, and from Young's and Poincaré's inequalities, it follows that
\begin{align}\label{a_1mulvar}
- a_1(t)\int_0^L u u_{t}\,dx &\leq \widehat{\epsilon}_1 b_{01}^2\Vert u \Vert^2+C_{\widehat{\epsilon}_1}\Vert u_t \Vert^2 \\
& \leq \widehat{\epsilon}_1 b_{01}^2\frac{L^2}{\pi^2}\Vert u_x \Vert^2+C_{\widehat{\epsilon}_1}\Vert u_t \Vert^2,\nonumber
\end{align}
\begin{align}\label{a_2mulpsi}
 - a_2(t)\int_0^L v v_{t}\,dx \leq  \widehat{\epsilon}_2 b_{02}^2\frac{L^2}{\pi^2}\Vert v_x \Vert^2+C_{\widehat{\epsilon}_2}\Vert v_t \Vert^2,
\end{align}and
\begin{align}\label{a_3mulw}
 - a_3(t)\int_0^L w w_{t}\,dx \leq  \widehat{\epsilon}_3 b_{03}^2\frac{L^4}{\pi^4}\Vert w_{xx} \Vert^2+C_{\widehat{\epsilon}_3}\Vert w_t \Vert^2
\end{align}for any $\widehat{\epsilon}_i>0$ and some constants $C_{\widehat{\epsilon}_i}>0$, $i=1, 2, 3$.

Thus, using \eqref{damping_123}, \eqref{eq:EnDiss3} and from \eqref{de_V_1_im}-\eqref{a_3mulw}, it follows that
\begin{equation}\label{fullequiLEder}
\begin{split}
\mathcal{L}^{\prime}(t) =& \mu E^{\prime}(t) +\sum_{i=0}^3\mu_i \mathcal{L}^{\prime}_i(t)\\
\leq &\left[-\left(\dfrac{\mu a_{01}}{\rho_1 h_1}-\mu_0\left(1+\frac{C_{\widehat{\epsilon}_1}}{\rho_1 h_1}\right)\right)\rho_1h_1\|u_t\|^2-\mu_0\left(1-\left(\epsilon_1+\frac{\widehat{\epsilon}_1 b_{01}^2L^2}{E_1h_1\pi^2}\right)\right)E_1h_1\|u_x\|^2\right.\\
&-\left(\dfrac{\mu a_{02}}{\rho_3 h_3}-\mu_0\left(1+\frac{C_{\widehat{\epsilon}_2}}{\rho_3 h_3}\right)\right)\rho_3h_3\|v_t\|^2 -\mu_0\left(1-\left(\epsilon_2+\frac{\widehat{\epsilon}_2 b_{02}^2L^2}{E_3h_3\pi^2}\right)\right)E_3h_3\|v_x\|^2 \\
& \left. -\left(\dfrac{\mu a_{03}}{\rho h}-\mu_0\left(1+\frac{C_{\widehat{\epsilon}_3}}{\rho h}\right)\right)\rho h \left\|w_t\right\|^2  -\mu_0\left(1-\left(\epsilon_3+\frac{\widehat{\epsilon}_3 b_{03}^2L^4}{E I\pi^4}\right)\right)E I\left\|w_{xx}\right\|^2 \right. \\
&\left.-\mu_0k\left\|-u+v+\alpha w_{x}\right\|^2 - \mu_0 \sum_{i=1}^3\dfrac{\lvert\beta_i\rvert}{2} \tau_i(t) \Vert z_i^2(\cdot, t)\Vert_{L^2(0,1)}^2 \right] +\sum_{i=1}^3W_i + \sum_{i=1}^3S_i
\end{split}
\end{equation}
where
\begin{equation*}
\begin{cases}
W_1=	\dfrac{1}{2} \left\langle \Pi_{\mu_0,\mu_1} (u_{t}(L,t), z_{1}(1, t)), (u_{t}(L,t), z_{1}(1, t)) \right\rangle,\\
\Pi_{\mu_0,\mu_1} =\mu\Phi^{\dot\tau_1(t)}_{\alpha_1,\beta_1}
	+ \mu_0 C_{\epsilon_1, E_1, h_1}\begin{pmatrix}
		\alpha_1^2 & \alpha_1\beta_1 \\
		\alpha_1\beta_1 & \beta_1^2
	\end{pmatrix}
	+ \mu_1\lvert\beta_1\rvert\begin{pmatrix}
		1 & 0 \\ 0 & 0
	\end{pmatrix},
\end{cases}
\end{equation*}
\begin{equation*}
\begin{cases}
W_2=	\dfrac{1}{2} \left\langle \Pi_{\mu_0,\mu_2} (v_{t}(L,t), z_{2}(1, t)), (v_{t}(L,t), z_{2}(1, t)) \right\rangle,\\
\Pi_{\mu_0,\mu_2} =\mu\Phi^{\dot\tau_2(t)}_{\alpha_2,\beta_2}
	+ \mu_0 C_{\epsilon_2, E_3, h_3}\begin{pmatrix}
		\alpha_2^2 & \alpha_2\beta_2 \\
		\alpha_2\beta_2 & \beta_2^2
	\end{pmatrix}
	+ \mu_2\lvert\beta_2\rvert\begin{pmatrix}
		1 & 0 \\ 0 & 0
	\end{pmatrix},
\end{cases}
\end{equation*}
\begin{equation*}
\begin{cases}
W_3=	\dfrac{1}{2} \left\langle \Pi_{\mu_0,\mu_3} (w_{tx}(L,t), z_{3}(1, t)), (w_{tx}(L,t), z_{3}(1, t)) \right\rangle,\\
\Pi_{\mu_0,\mu_3} =\mu\Phi^{\dot\tau_3(t)}_{\alpha_3,\beta_3}
	+ \mu_0 C_{\epsilon_3, E,  I}\begin{pmatrix}
		\alpha_3^2 & \alpha_3\beta_3 \\
		\alpha_3\beta_3 & \beta_3^2
	\end{pmatrix}
	+ \mu_3\lvert\beta_3\rvert\begin{pmatrix}
		1 & 0 \\ 0 & 0
	\end{pmatrix},
\end{cases}
\end{equation*}and
\begin{equation*}
S_i = -\mu_i\dfrac{\lvert\beta_i\rvert}{2}\int_0^1(1-\dot\tau_i(t)\rho)z_i^2(\rho, t) d\rho + \mu_0\dfrac{\lvert\beta_i\rvert}{2} \tau_i(t) \Vert z_i^2(\cdot, t)\Vert_{L^2(0,1)}^2, \,\,  i=1, 2, 3.
\end{equation*}

Our task now is to prove that
$$\mathcal{L}^{\prime}(t)\leq -\lambda E(t),$$
for some positive constant $\lambda$. To do so, let us analyze each term $W_i$ and  $S_i$  in \eqref{fullequiLEder}, for $i=1,2,3$.

\vspace{0.2cm}

\noindent\textbf{Estimate for $W_i$:} Since the matrices $\Phi^{\dot{\tau}_i(t)}_{\alpha_i,\beta_i}$ are negative definite, it follows from the continuity of the trace and determinant functions that one can choose $\mu_0, \mu_i \in (0,1)$ sufficiently small so that the perturbed matrix $\Pi_{\mu_0,\mu_i}$ remains negative definite. Thus,
\begin{align}\label{W_inegativeforesta}
  W_i \leq0, \, \text{for } i=1,2,3.
\end{align}
\vspace{0.2cm}
\noindent\textbf{Estimate for $S_i$:} By using \eqref{eq:TauCond}, we get
\begin{equation*}
	\begin{aligned}
		S_i\leq& -\dfrac{\mu_i\lvert \beta_i \rvert }{2}(1-d_i)\Vert z_i^2(\cdot, t)\Vert_{L^2(0,1)}^2+ \mu_0\dfrac{\lvert\beta_i\rvert}{2} M_i \Vert z_i^2(\cdot, t)\Vert_{L^2(0,1)}^2\\
		\leq& -\left ( \mu_i(1-d_i)-\mu_0 M_i\right)\dfrac{\lvert\beta_i\rvert}{2}\Vert z_i^2(\cdot, t)\Vert_{L^2(0,1)}^2.
	\end{aligned}
\end{equation*}
So, choosing
\begin{equation*}
	 \dfrac{\mu_0 M_i}{(1-d_i)}<\mu_i,
\end{equation*}
we have that
\begin{align}\label{S_inegativeforestabi}
S_i<0, \, \, i=1, 2, 3.
\end{align}
Thus, let us choose $\epsilon_i, \widehat{\epsilon}_i>0$, $i=1,2, 3$,  $\mu_0>0$ sufficiently small and $\mu>0$ such that
\begin{align*}
\max \left\{ \epsilon_1+\frac{\widehat{\epsilon}_1 b_{01}^2L^2}{E_1h_1\pi^2}, \epsilon_2+\frac{\widehat{\epsilon}_2 b_{02}^2L^2}{E_3h_3\pi^2}, \epsilon_3+\frac{\widehat{\epsilon}_3 b_{03}^2L^4}{E I\pi^4} \right\} < 1,
\end{align*}
\begin{align*}
  \dfrac{\mu_0 M_i}{1-d_i} < \mu_i, \quad i=1,2,3,
\end{align*}and
\begin{align*}
 \max\left\{ \mu_4, \mu_5\right\}< \mu,
\end{align*}where
\begin{align*}
\mu_5 = \max\left\{ \frac{\mu_0}{a_{01}}(\rho_1h_1+C_{\widehat{\epsilon}_1}), \frac{\mu_0}{a_{02}}(\rho_3h_3+C_{\widehat{\epsilon}_2}),\frac{\mu_0}{a_{03}}(\rho h+C_{\widehat{\epsilon}_3})\right\},
\end{align*}and $\mu_4$ is given by \eqref{mu4}.

Next, consider $\lambda > 0$ satisfying
\begin{align*}
\lambda \le \min &\left\{ 2\left(\dfrac{\mu a_{01}}{\rho_1 h_1}-\mu_0\left(1+\frac{C_{\widehat{\epsilon}_1}}{\rho_1 h_1}\right)\right), 2\left(\dfrac{\mu a_{02}}{\rho_3 h_3}-\mu_0\left(1+\frac{C_{\widehat{\epsilon}_3}}{\rho_3 h_3}\right)\right), 2\left(\dfrac{\mu a_{03}}{\rho h}-\mu_0\left(1+\frac{C_{\widehat{\epsilon}_3}}{\rho h}\right)\right), \right. \\
&\left. 2\mu_0\left(1-\left(\epsilon_1+\frac{\widehat{\epsilon}_1 b_{01}^2L^2}{E_1h_1\pi^2}\right)\right), 2\mu_0\left(1-\left(\epsilon_2+\frac{\widehat{\epsilon}_2 b_{02}^2L^2}{E_3h_3\pi^2}\right)\right), 2\mu_0\left(1-\left(\epsilon_3+\frac{\widehat{\epsilon}_3 b_{03}^2L^4}{E I \pi^4}\right)\right)\right\}.
\end{align*}
Then, from \eqref{fullequiLEder}--\eqref{S_inegativeforestabi} and using inequality \eqref{equinerlia}, it follows that $\lambda>0$ satisfies
\begin{align*}
\mathcal{L}'(t) \le -\lambda E(t) \le -\dfrac{\lambda}{\mu_*} \mathcal{L}(t),
\end{align*}
where $\mu_*=\mu+\mu_4$, and consequently,
\begin{align}\label{ener_equi_lia_decay}
\mathcal{L}(t) \le e^{-\dfrac{\lambda}{\mu_*}t} \mathcal{L}(0).
\end{align}

Finally, combining \eqref{equinerlia} and \eqref{ener_equi_lia_decay}, there exist positive constants $\lambda$, $\zeta = \frac{\mu+\mu_4}{\mu-\mu_4}$ and $\mu_*=\mu+\mu_4$ such that
\[
E(t) \le \zeta \, e^{-\dfrac{\lambda}{\mu_*} t} E(0), \quad \text{for all } t \ge 0.
\]
This completes the proof of Theorem \ref{th:Lyapunov0}.
\end{proof}


\section{Well-posedness for the control system}\label{sec4}
In this section, we establish the well-posedness of system \eqref{2bbm}.  We consider both the homogeneous case, where the boundary inputs $f_1$, $f_2$, and $f_3$ vanish, and the nonhomogeneous case, in which these three controls are present.

\subsection{The homogeneous system}
We begin by considering the following homogeneous system:
\begin{equation}\label{2bbm-hom}
\left\{\begin{array}{ll}
\rho_{1 } h_1 u_{tt}-E_1h_1u_{xx}-k\left(  -u+v+\alpha w_{x}\right)   =0, &x\in(0,L),\,\,\,
t>0,\\
\rho_{3} h_3 v_{tt}-E_3 h_3 v_{xx}+k\left( -u+v+\alpha w_{x}\right)
=0,  &x\in(0,L),\,\,\,
t>0,\\
\rho h w_{tt}+EI w_{xxxx}-\alpha k \left( -u+v+\alpha w_{x}\right)_x
=0, & x\in(0,L),\,\,\, t>0,\\
u(0,t)=v(0,t)=w_x(0, t)=w_{xxx}(0,t) =0, & t>0,\\
w_{xx}(L,t)=w_{xxx}(L,t) =0, & t>0,\\
u_{tt}(L, t)+u_{x}(L, t)= 0, &
t>0,\\
v_{tt}(L, t) + v_x(L, t) =0, &
t>0,\\
w_{tt}(L, t)-u(L,t)+v(L, t)+\alpha w_x(L,t)=0, &
t>0,\\
\left(  u, v, w \right)  \left( x,0 \right)  =\left(
u_{0}, v_{0}, w_{0}\right) ( x ), & x\in(0,L),\\
\left( u_{t}, v_{t}, w_{t}\right)  \left(  x,0\right)  =\left( u_{1}, v_{1}, w_{1}\right)  \left(  x\right), & x\in(0,L).
\end{array}\right.\end{equation}
Observe that, by a formal computation, we find that a solution of the system \eqref{2bbm-hom} satisfies the following identity
\begin{equation*}
\begin{aligned}
\dfrac{1}{2} \dfrac{d}{d t} & \bigg[\rho_1 h_1\left\|u_t\right\|_{L^2}^2+E_1 h_1\left\|u_x\right\|_{L^2}^2+\rho_3 h_3\left\|v_t\right\|_{L^2}^2+E_3 h_3\left\|v_x\right\|_{L^2}^2+\rho  h\left\|w_t\right\|_{L^2}^2+E I\left\|w_{x x}\right\|_{L^2}^2
\\ &
+k \|-u+ v+\alpha w_x \|_{L^2}^2  +E_1 h_1\left|u_t(L)\right|^2+E_3 h_3\left|v_t(L)\right|^2+\alpha k\left|w_t(L)\right|^2\bigg]=0.
\end{aligned}
\end{equation*}
With this in mind, we introduce the new variables
$$
\Psi_1=u_{t},\ \ \ \Psi_2=v_{t},\ \ \ \Psi_3=w_{t},\ \ \ \Psi_4(\cdot)=\Psi_1(L, \cdot), \ \ \ \Psi_5(\cdot)=\Psi_2(L, \cdot), \ \ \ \Psi_6(\cdot)=\Psi_3(L, \cdot),
$$
and define the vector functions
$$
U=(u,\Psi_1, v,\Psi_2,w,\Psi_3, \Psi_4, \Psi_5, \Psi_6)^{\top}\text{ \ and \ }U_{0}=\left(  u_{0}%
,u_{1},v_{0},v_{1},w_{0},w_{1}, \Psi_4(0), \Psi_5(0), \Psi_6(0)\right)  ^{\top}.%
$$
Hence, the system \eqref{2bbm-hom} can be reformulated as an abstract Cauchy problem
\begin{equation}\label{abs}
\left\{
\begin{array}
[c]{l}
U_t = \mathcal{P} U\\
U(0) = U_0,
\end{array}\right.
\end{equation} where the operator $\mathcal{P}$ is given by
\begin{equation}\label{opae2}
\mathcal{P}U:=\begin{pmatrix}
\Psi_1 \\
\dfrac{1}{\rho_1 h_1} \left[ E_1 h_1 u_{xx} + k(-u + v + \alpha w_x) \right] \\
\Psi_2 \\
\dfrac{1}{\rho_3 h_3} \left[ E_3 h_3 v_{xx} - k(-u + v + \alpha w_x) \right] \\
\Psi_3 \\
\dfrac{1}{\rho h} \left[ -EI w_{xxxx} + \alpha k (-u + v + \alpha w_x)_x \right] \\
- u_x(L,\cdot) \\
- v_x(L,\cdot) \\
- \left[ -u(L,\cdot) + v(L,\cdot) + \alpha w_x(L,\cdot) \right]
\end{pmatrix}.
\end{equation}

Furthermore, the domain of $\mathcal{P}: D\left(\mathcal{P}\right) \subset \mathcal{H}_2 \rightarrow \mathcal{H}_2$ is defined by
\begin{equation}\label{DAA-1}
D(\mathcal{P})=
\left\lbrace
	\begin{aligned}
		&U=(u,\Psi_1, v,\Psi_2,w,\Psi_3, \Psi_4, \Psi_5, \Psi_6)^{\top}\in \mathcal{H}_2; \\ & (u,v,w)\in \left[ H^2(0,L)\cap H_*^1(0,L) \right]^2\times H^4_{*}(0,L),  \\
	& (\Psi_1, \Psi_2,\Psi_3)\in \left[  H_*^1(0,L) \right]^2\times H_*^2(0,L),
    \\&\Psi_4=\Psi_1(L), \ \Psi_5=\Psi_2(L),  \ \Psi_6=\Psi_3(L)
	\end{aligned}
\right\rbrace.
\end{equation}

Now, let us denote by  $\rho(\mathcal{P})$ the resolvent set of the operator $\mathcal{P}$. Then, we have the following result.
\begin{lemma}\label{0inresolv}
  Let $\mathcal{H}_2$ and $\left(
D\left(\mathcal{P}\right), \mathcal{P}\right)$ be defined as before. Then, $0 \in \rho(\mathcal{P})$. Moreover, $\mathcal{P}^{-1}$ is compact.
\end{lemma}
\begin{proof}
For $G=\left(g_1, g_2, g_3, g_4, g_5, g_6, g_7, g_8, g_9 \right) \in \mathcal{H}_2$, we show the existence of $U \in D\left(\mathcal{P}\right),$ unique solution of the equation
$$
\mathcal{P} U=G.
$$
Equivalently, one must consider the system given by
\begin{align}
& \Psi_1= g_1, \quad \Psi_2=g_3, \quad \Psi_3=g_5, \label{eq:system1_1}\\
& E_1h_1u_{xx}+k\left(-u+v+\alpha w_x\right)=\rho_1 h_1 g_2, \label{eq:system2_1} \\
& E_3h_3v_{xx}-k\left(-u+v+\alpha w_x\right)=\rho_3 h_3 g_4, \label{eq:system3_1}\\
& E I w_{xxxx}+\alpha k\left(-u+v+\alpha w_x\right)_x=\rho h g_6, \label{eq:system4_1}\\
& -u_x(L)= g_7,\label{eq:system5_1}\\
&-v_x(L)=g_8, \label{eq:system6_1}\\
&-[-u(L)+v(L)+\alpha w_x(L)]=g_9, \label{eq:system7_1}
\end{align}with the following boundary conditions
\begin{equation}\label{eq:system_cond1}
u(0)=v(0)=w_x(0)=w_{xxx}(0)=w_{xx}(L)=w_{xxx}(L)=0.
\end{equation}

Let $(u^{*},v^{*},w^{*}) \in \left[ H_*^1(0,L) \right]^2 \times H_*^2(0,L)$.  We multiply the equations \eqref{eq:system2_1}, \eqref{eq:system3_1}, and \eqref{eq:system4_1}  by $u^{*}$, $v^{*}$, and $w^{*}$, respectively, and integrate by parts over $(0,L)$.  Using \eqref{eq:system5_1}, \eqref{eq:system6_1}, \eqref{eq:system7_1} and the boundary conditions in \eqref{eq:system_cond1}, and then summing the resulting expressions, we obtain
\begin{equation} \label{eq:bilinear_def0}
\mathcal{B}((u,v,w),(u^{*},v^{*},w^{*})) = \mathcal{L}(u^{*},v^{*},w^{*}),
\quad \forall (u^{*},v^{*},w^{*}) \in \left[ H_*^1(0,L) \right]^2 \times H_*^2(0,L),
\end{equation}
where
\begin{equation} \label{eq:bilinear_def1}
\begin{aligned}
\mathcal{B}((u,v,w),(u^{*},v^{*},w^{*}))
&= E_1 h_1 \langle u_x, u_x^{*} \rangle + E_3 h_3 \langle v_x, v_x^{*} \rangle + EI \langle w_{xx}, w_{xx}^{*} \rangle \\
&\quad + k \langle -u+v+\alpha w_x, -u^{*}+v^{*}+\alpha w_x^{*} \rangle,
\end{aligned}
\end{equation}
and
\begin{align} \label{eq:bilinear_def2}
\mathcal{L}(u^{*},v^{*},w^{*})
&= -\rho_1 h_1 \langle g_2, u^{*} \rangle - \rho_3 h_3 \langle g_4, v^{*} \rangle - \rho h \langle g_6, w^{*} \rangle \\
&\quad - E_1 h_1 g_7 u^{*}(L) - E_3 h_3 g_8 v^{*}(L) - \alpha k g_9 w^{*}(L). \nonumber
\end{align}

From \eqref{eq:bilinear_def1} and \eqref{eq:bilinear_def2}, it follows that the bilinear form
\[
\mathcal{B}: \left[ \left[ H_*^1(0,L) \right]^2 \times H_*^2(0,L) \right]^2 \to \mathbb{C}
\]
is continuous and coercive, and the linear functional
\[
\mathcal{L}: \left[ H_*^1(0,L) \right]^2 \times H_*^2(0,L) \to \mathbb{C}
\]
is continuous as well. Therefore, by the Lax–Milgram theorem, there exists a unique weak solution
\[
(u,v,w) \in \left[ H_*^1(0,L) \right]^2 \times H_*^2(0,L)
\]
to the variational problem \eqref{eq:bilinear_def0}.

Furthermore, using \eqref{eq:system1_1} and standard elliptic regularity results, we conclude that $U \in D(\mathcal{P})$,
which shows that the operator $\mathcal{P}$ is bijective.
Following arguments similar to \cite[Proposition 2]{CQS}, one can also verify that $\mathcal{P}^{-1}$ is bounded,
implying that $0 \in \rho(\mathcal{P})$, the resolvent set of $\mathcal{P}$.
Finally, by the Sobolev embedding theorem, the natural inclusion
\[
i: D(\mathcal{P}) \hookrightarrow \mathcal{H}_2
\]
is compact, which ensures that $\mathcal{P}^{-1}$ is a compact operator.\end{proof}

The next result ensures that the operator $\mathcal{P}$ generates a group. Precisely, we have the following.

\begin{proposition}\label{teogeneinfiA}
The operator $\left(D\left(\mathcal{P}\right), \mathcal{P}\right)$ is infinitesimal generator of a group of isometries $\{S(t)\}_{t\in\mathbb{R}}$ in $\mathcal{H}_2$.
\end{proposition}

\begin{proof}
Let $U^{\sharp}\in D\left(\mathcal{P}\right)$ given by \eqref{U-r-a}. Then, for all $U\in D\left(\mathcal{P}\right)$, given by \eqref{U-r}, we obtain, after some integration by parts and rearrangements, that
\begin{equation*}
\begin{split}
\left\langle \mathcal{P}U, U^{\sharp} \right\rangle_{\mathcal{H}_2}
=& \ \ \rho_1h_1\left\langle \dfrac{1}{\rho_1 h_1} \left[ E_1 h_1 u_{xx} + k(-u + v + \alpha w_x) \right], \Psi_1^{\sharp}\right\rangle+E_1 h_1\left\langle \Psi_{1x}, u^{\sharp}_x\right\rangle\\&+\rho_3 h_3\left\langle \dfrac{1}{\rho_3 h_3} \left[ E_3 h_3 v_{xx} - k(-u + v + \alpha w_x) \right], \Psi_2^{\sharp}\right\rangle+E_3 h_3\left\langle \Psi_{2x}, v_x^{\sharp}\right\rangle \\
& +\rho h\left\langle \dfrac{1}{\rho h} \left[ -EI w_{xxxx} + \alpha k (-u + v + \alpha w_x)_x \right], \Psi_3^{\sharp}\right\rangle +EI \left\langle \Psi_{3xx}, w_{xx}^{\sharp}\right\rangle \\ &+k\left\langle-\Psi_1+\Psi_2+\alpha \Psi_{3x},-u^{\sharp}+v^{\sharp}+\alpha w_{x}^{\sharp}   \right\rangle-E_1 h_1\left(u_x(L),\Psi_4^{\sharp}   \right)_{\mathbb{R}}\\
&-E_3h_3\left(v_x(L),\Psi_5^{\sharp}   \right)_{\mathbb{R}}-\alpha k\left([-u(L)+v(L)+\alpha k w_{x}(L)],\Psi_6^{\sharp}   \right)_{\mathbb{R}}\\
 & =\rho_1h_1\left\langle\Psi_1, - \dfrac{1}{\rho_1 h_1} \left[ E_1 h_1 u_{xx}^{\sharp} + k(-u^{\sharp} + v^{\sharp} + \alpha w_x^{\sharp}) \right]\right\rangle+E_1 h_1\left\langle u_x,-\Psi_{1x}^{\sharp}\right\rangle\\&+\rho_3 h_3\left\langle \Psi_2,-\dfrac{1}{\rho_3 h_3} \left[ E_3 h_3 v_{xx}^{\sharp} - k(-u^{\sharp} + v^{\sharp} + \alpha w_x^{\sharp}) \right] \right\rangle+E_3 h_3\left\langle v_x,-\Psi_{2x}^{\sharp} \right\rangle \\
& +\rho h\left\langle \Psi_3,-\dfrac{1}{\rho h} \left[ -EI w_{xxxx}^{\sharp} + \alpha k (-u^{\sharp} + v^{\sharp} + \alpha w_x^{\sharp})_x \right] \right\rangle+EI \left\langle  w_{xx},-\Psi_{3xx}^{\sharp}\right\rangle \\ &+k\left\langle-u+v+\alpha w_{x} ,-\left[\Psi_1^{\sharp}+\Psi_2^{\sharp}+\alpha \Psi_{3x}^{\sharp} \right]  \right\rangle+E_1 h_1\left(\Psi_4,u_x^{\sharp}(L)   \right)_{\mathbb{R}}\\
&+E_3h_3\left(\Psi_5,v_x^{\sharp}(L)   \right)_{\mathbb{R}}+\alpha k\left(\Psi_6,[-u^{\sharp}(L)+v^{\sharp}(L)+\alpha k w_{x}^{\sharp}(L)]   \right)_{\mathbb{R}}\\ &=\left\langle U, -\mathcal{P}U^{\sharp} \right\rangle_{\mathcal{H}_2}.
\end{split}
\end{equation*}
Hence, $D\left(\mathcal{P}\right) \subset D\left(\mathcal{P}^{*}\right)$ and $\mathcal{P}^{*}U^{\sharp} = -\mathcal{P}U^{\sharp}$, for any $U^{\sharp} \in D\left(\mathcal{P}\right)$.
\noindent On the other hand, let $U^{\sharp}
\in D\left(\mathcal{P}^{*}\right)$. Then, by definition, 
\begin{align}\label{adj_1}
\left\langle
\mathcal{P}U, U^{\sharp}\right\rangle_{\mathcal{H}_2} = \left\langle
U, \mathcal{P}^{*}U^{\sharp}\right\rangle_{\mathcal{H}_2}, \text{ for all }\ \ U
\in D\left(\mathcal{P}\right).
\end{align}By the same calculations as above, we have that
\begin{equation}\label{adj_2}
\begin{aligned}
\left\langle
U,\mathcal{P}^{*}U^{\sharp}\right\rangle_{\mathcal{H}_2}  =   &\left\langle
\mathcal{P}U, U^{\sharp}\right\rangle_{\mathcal{H}_2}
\end{aligned}
\end{equation}
Therefore, from \eqref{adj_1} and \eqref{adj_2} we obtain that
\begin{align*}
   0=& - E_1h_1\left(u_x(0) , \Psi_1^{\sharp}(0) \right)_{\mathbb{R}}-E_3h_3\left((v_x(0) , \Psi_2^{\sharp}(0) \right)_{\mathbb{R}}-EI\left((w_{xx}(0) , \Psi_{3x}^{\sharp}(0) \right)_{\mathbb{R}} \\
&+E_1h_1\left\langle u_x(L), \Psi_1^{\sharp}(L)-\Psi_4^{\sharp}   \right\rangle+E_3h_3\left\langle v_x(L), \Psi_2^{\sharp}(L)-\Psi_5^{\sharp}   \right\rangle\nonumber \\
&-EI\left(w_{xx}(0) , \Psi_{3x}^{\sharp}(0)   \right)_{\mathbb{R}} +EI\left\langle \Psi_{3x}(L), w_{xx}^{\sharp}(L)  \right\rangle-EI\left\langle \Psi_{3}(L), w_{xxx}^{\sharp}(L)  \right\rangle
\\&
+\alpha k\left([-u(L)+v(L)+\alpha k w_{x} (L)],\Psi_3^{\sharp}(L)-\Psi_6^{\sharp}   \right)_{\mathbb{R}}. \nonumber
\end{align*}
Since the above equality holds for all $U \in D(\mathcal{P})$, it follows that $U^{\sharp} \in D(\mathcal{P})$. Consequently, we deduce that $D(\mathcal{P}) = D(\mathcal{P}^{*})$, and from the previous computations, we also obtain $\mathcal{P}^{*} = -\mathcal{P}$. Therefore, $\mathcal{P}$ is a skew-adjoint operator, and by Stone's Theorem, it generates a group of isometries $\{S(t)\}_{t \in \mathbb{R}}$ on $\mathcal{H}_2$.
\end{proof}

As an immediate consequence of Lemma \ref{0inresolv}, Proposition \ref{teogeneinfiA}, and standard results from the theory of evolution equations (see, {\em e.g.,}  \cite{Pazy}), we obtain the following existence and uniqueness result.
\begin{theorem}\label{ex-hom}
For any $U_0 \in D\left(\mathcal{P}\right)$, there exists a unique solution $U$ of the abstract Cauchy problem \eqref{abs} such that
\begin{align}\label{classSoluhomo}
U  \in C\left([0, \infty); D\left(\mathcal{P}\right) \right) \cap C^1\left([0, \infty); \mathcal{H}_2\right).
\end{align}
Moreover, for any $t>0$, we have that
\begin{align}\label{identi_gisometr}
\Vert U(t) \Vert^2_{\mathcal{H}_2} =\Vert U_0 \Vert^2_{\mathcal{H}_2}.
\end{align}
\end{theorem}

\subsection{The nonhomogeneous system} In this subsection, we focus on analyzing the complete system \eqref{2bbm}. We start with the following result:
 \begin{proposition}
For any $U_0 \in D\left(\mathcal{P}\right)  $
and $ f_i \in  C^{\infty}_0(0,\infty), \ \ i=1,2,3  $,
system \eqref{2bbm} has a unique solution  $U  \in C\left([0, \infty); D\left(\mathcal{P}\right) \right) \cap C^1\left([0, \infty); \mathcal{H}_2\right).$
\end{proposition}

\begin{proof}
Let $\phi_1 \in C^{\infty}[0, L]$ be a smooth function satisfying the boundary conditions $\phi_1(0) = \phi_1(L) = 0$ and $\phi_1'(L) = -1$.
We now introduce the following change of variables:
\begin{equation}\label{change}
\left(\begin{array}{c}
z \\ \eta \\ \varphi \end{array}\right) = \left(\begin{array}{c}
u \\ v \\ w \end{array}\right) - \left(\begin{array}{c}
\tilde{u} \\ \tilde{v} \\ \tilde{w} \end{array}\right) + \left(\begin{array}{c}
f_1(t)\phi_1(x) \\ f_2(t)\phi_1(x) \\ f_3(t)\phi_1(x) \end{array}\right),
\end{equation}
where $ \left(\tilde{u}, \tilde{v}, \tilde{w}\right)$ is the unique  solution of the system
\begin{equation}\label{siste22}
\left\{\begin{array}{ll}
\rho_{1} h_1 \tilde{u}_{tt} - E_1 h_1 \tilde{u}_{xx} - k\left( -\tilde{u} + \tilde{v} + \alpha \tilde{w}_x \right) = 0, & x \in (0,L),\;\; t > 0, \\[1ex]
\rho_{3} h_3 \tilde{v}_{tt} - E_3 h_3 \tilde{v}_{xx} + k\left( -\tilde{u} + \tilde{v} + \alpha \tilde{w}_x \right) = 0, & x \in (0,L),\;\; t > 0, \\[1ex]
\rho h \tilde{w}_{tt} + EI \tilde{w}_{xxxx} - \alpha k \left( -\tilde{u} + \tilde{v} + \alpha \tilde{w}_x \right)_x = 0, & x \in (0,L),\;\; t > 0, \\[1ex]
\tilde{u}(0,t) = \tilde{v}(0,t) = \tilde{w}_x(0,t) = \tilde{w}_{xxx}(0,t) = 0, & t > 0, \\[1ex]
\tilde{w}_{xx}(L,t) = \tilde{w}_{xxx}(L,t) = 0, & t > 0, \\[1ex]
\tilde{u}_{tt}(L,t) + \tilde{u}_x(L,t) = 0, & t > 0, \\[1ex]
\tilde{v}_{tt}(L,t) + \tilde{v}_x(L,t) = 0, & t > 0, \\[1ex]
\tilde{w}_{tt}(L,t) - \tilde{u}(L,t) + \tilde{v}(L,t) + \alpha \tilde{w}_x(L,t) = 0, & t > 0, \\[1ex]
\left( \tilde{u}, \tilde{v}, \tilde{w} \right)(x,0) = \left( u_0, v_0, w_0 \right)(x), & x \in (0,L), \\[1ex]
\left( \tilde{u}_t, \tilde{v}_t, \tilde{w}_t \right)(x,0) = \left( u_1, v_1, w_1 \right)(x), & x \in (0,L),
\end{array}
\right.
\end{equation}
given by Theorem \ref{ex-hom}. Thus, $ \left(z, \eta, \varphi\right)$ solves the problem
\begin{equation}\label{siste23}
\left\{
\begin{array}{ll}
\rho_{1} h_1 z_{tt} - E_1 h_1 z_{xx} - k\left( -z + \eta + \alpha \varphi_x \right) = F, & x \in (0,L),\;\; t > 0, \\[1ex]
\rho_{3} h_3 \eta_{tt} - E_3 h_3 \eta_{xx} + k\left( -z + \eta + \alpha \varphi_x \right) = G, & x \in (0,L),\;\; t > 0, \\[1ex]
\rho h \varphi_{tt} + EI \varphi_{xxxx} - \alpha k \left( -z + \eta + \alpha \varphi_x \right)_x = H, & x \in (0,L),\;\; t > 0, \\[1ex]
z(0,t) = \eta(0,t) = \varphi_x(0,t) = \varphi_{xxx}(0,t) = 0, & t > 0, \\[1ex]
\varphi_{xx}(L,t) = \varphi_{xxx}(L,t) = 0, & t > 0, \\[1ex]
z_{tt}(L,t) + z_x(L,t) = 0, & t > 0, \\[1ex]
\eta_{tt}(L,t) + \eta_x(L,t) = 0, & t > 0, \\[1ex]
\varphi_{tt}(L,t) - z(L,t) + \eta(L,t) + \alpha \varphi_x(L,t) = 0, & t > 0, \\[1ex]
\left( z, \eta, \varphi \right)(x,0) = \left( 0, 0, 0 \right), & x \in (0,L), \\[1ex]
\left( z_t, \eta_t, \varphi_t \right)(x,0) = \left( 0, 0, 0 \right), & x \in (0,L).
\end{array}
\right.
\end{equation}
Here,
\begin{equation*}
 \left\{\begin{array}{ll}
F &= \rho_1 h_1 f_1^{\prime \prime}(t)\phi_1(x) - E_1 h_1 f_1(t)\phi_1^{\prime \prime}(x) - k\left[ -f_1(t)\phi_1(x) + f_2(t)\phi_1(x) + \alpha f_3(t)\phi_1^{\prime}(x) \right], \\ \\
G &= \rho_3 h_3 f_2^{\prime \prime}(t)\phi_1(x) - E_3 h_3 f_2(t)\phi_1^{\prime \prime}(x) + k\left[ -f_1(t)\phi_1(x) + f_2(t)\phi_1(x) + \alpha f_3(t)\phi_1^{\prime}(x) \right], \\ \\
H &= \rho h f_3^{\prime \prime}(t)\phi_1(x) + EI f_3(t)\phi_1^{(4)}(x) - \alpha k\left[ -f_1(t)\phi_1^{\prime}(x) + f_2(t)\phi_1^{\prime}(x) + \alpha f_3(t)\phi_1^{\prime \prime}(x) \right],
 \end{array}\right.
\end{equation*}
and
$$ (F(t,x),G(t,x),H(t, x))\in \left[C^1(\left[ 0,\infty \right]\times \left[ 0, L \right] )\right]^3.$$

Now, let us use the notation established in the previous section to express the system \eqref{siste23} as an abstract evolution equation in the following form
\begin{equation}\label{sistem24}
\left\{
\begin{array}
[c]{l}
W_t + \mathcal{P} W= \mathcal{K}\\
W(0) = 0,
\end{array}\right.\nonumber
\end{equation}
where $W = (z, z_t, \eta, \eta_t, \varphi, \varphi_t, z_t(L, T), \eta_t(L, t), \varphi_t(L, t) )^{\top}$ and
$$ \mathcal{K} = \left(0, \dfrac{F}{\rho_1 h_1}, 0, \dfrac{G}{\rho_3 h_3}, 0, \dfrac{H}{\rho h}, 0, 0, 0 \right)^{\top} \in\left[
C^1(\left[ 0,\infty \right]\times \left[ 0, L \right] )\right]^9.$$
Since the operator $\mathcal{P}$ generates a group of isometries on the Hilbert space $\mathcal{H}_2$, it follows that system \eqref{siste23} admits a unique solution
\[ W \in C\left([0, \infty); D\left(\mathcal{P}\right) \right) \cap C^1\left([0, \infty); \mathcal{H}_2\right). \]
Therefore, by reverting the change of variables in \eqref{change}, the proof is complete.
\end{proof}

With the well-posedness result established, we now proceed to analyze solutions of system \eqref{2bbm} in the transposition sense. For that, given $T > 0$, $U_0 \in D\left(\mathcal{P}\right)$, $$(J_1, J_2, J_3) \in  L^2\left(0,T;  \left(\left[ H_*^1(0,L) \right]^2\times H_*^2(0,L)\right)^{*}\right)$$
and $ f_i \in  L^2(0,T)$, for $i=1,2,3,$ consider the non-homogeneous system
\begin{equation}\label{2bbm_nonhomo_2}
\left\{\begin{array}{ll}
\rho_{1} h_1 u_{tt} - E_1 h_1 u_{xx} - k\left( -u + v + \alpha w_x \right) = J_1, & x \in (0,L),\;\; t > 0, \\[1ex]
\rho_{3} h_3 v_{tt} - E_3 h_3 v_{xx} + k\left( -u + v + \alpha w_x \right) = J_2, & x \in (0,L),\;\; t > 0, \\[1ex]
\rho h w_{tt} + EI w_{xxxx} - \alpha k \left( -u + v + \alpha w_x \right)_x = J_3, & x \in (0,L),\;\; t > 0, \\[1ex]
u(0,t) = v(0,t) = w_x(0,t) = w_{xxx}(0,t) = 0, & t > 0, \\[1ex]
w_{xx}(L,t) = w_{xxx}(L,t) = 0, & t > 0, \\[1ex]
u_{tt}(L,t) + u_x(L,t) = f_1, & t > 0, \\[1ex]
v_{tt}(L,t) + v_x(L,t) = f_2, & t > 0, \\[1ex]
w_{tt}(L,t) - u(L,t) + v(L,t) + \alpha w_x(L,t) = f_3, & t > 0, \\[1ex]
\left( u, v, w \right)(x,0) = \left( u_0, v_0, w_0 \right)(x), & x \in (0,L), \\[1ex]
\left( u_t, v_t, w_t \right)(x,0) = \left( u_1, v_1, w_1 \right)(x), & x \in (0,L).
\end{array}
\right.
\end{equation}

Now, let us introduce the following definition.
\begin{definition}
A solution by transposition to the system \eqref{2bbm_nonhomo_2} is defined as a vector-valued function
$U = (u, u_t, v, v_t, w, w_t, u_t(L), v_t(L), w_t(L))^\top$ belonging to the space
$C([0,T]; \mathcal{D}(\mathcal{P}))$, which satisfies the following property: For every
$\tau \in [0,T]$ and every  function $W^\tau \in \mathcal{D}(\mathcal{P}^*)$ of the form
\[
W^\tau = (\chi_0^\tau, \chi_1^\tau, \eta_0^\tau, \eta_1^\tau, \Theta_0^\tau, \Theta_1^\tau, \chi_1^\tau(L), \eta_1^\tau(L), \Theta_1^\tau(L))^\top,
\]
the following integral identity holds:
\begin{align}\label{equa2}
& \displaystyle \left\langle  \left( \begin{array}{c}
\rho_1 h_1 u_t(\cdot, \tau)\\ - \rho_1 h_1 u(\cdot, \tau) \\  \rho_3 h_3 v_t(\cdot, \tau)\\ - \rho_3h_3 v(\cdot, \tau) \\
 \rho h w_t(\cdot, \tau) , -\rho h w(\cdot, \tau)\\
 -E_1 h_1 u(L, \tau)\\
 -E_3 h_3 v(L, \tau)\\
 -\alpha k w(L, \tau) \end{array}
 \right)^{\top}, \left(\begin{array}{c}
 \chi_0^{\tau}\\ \chi_1^{\tau} \\ \eta_0^{\tau} \\ \eta_1^{\tau}\\ \Theta_0^{\tau} \\ \Theta_1^{\tau} \\\chi_t(L, \tau)\\ \eta_t(L, \tau)\\ \Theta_t(L, \tau) \end{array}\right)  \displaystyle \right\rangle
 =\displaystyle \left\langle  \left(\begin{array}{c}
 \rho_1 h_1 u_1 \\ - \rho_1 h_1 u_0 \\  \rho_3 h_3 v_1\\- \rho_3 h_3 v_0 \\
 \rho h w_1 \\ -\rho h w_0\\
 - E_1 h_1 u(L, 0)\\
 -E_3 h_3 v(L, 0)\\
 -\alpha k w(L, 0)\end{array}\right), \left(\begin{array}{c}
 \chi(0)\\ \chi_t(0) \\ \eta(0) \\ \eta_t(0)\\ \Theta(0) \\ \Theta_t(0) \\\chi_t(L, 0)\\ \eta_t(L, 0)\\ \Theta_t(L, 0) \end{array}\right)  \displaystyle \right\rangle
 \\
& +\displaystyle\int_{0}^{\tau} \, \left \langle \left(J_1(t), J_2(t), J_3(t)) \right), \left(\chi(t), \eta(t), \Theta(t) \right)\right\rangle_{\left[\left(  H_*^1(0,L) \right)^{*}, H_*^1(0,L) \right]^3}dt \nonumber\\
&+\left((E_1h_1f_1(t), E_3 h_3 f_2(t), \alpha k f_3(t)),  1_{(0, \tau)}(\chi(L, t), \eta(L, t), \Theta(L, t) )\right)_{[L^2(0, T)]^3} \nonumber \\
&+ \left( \left(E_1h_1u_t(L, 0), E_3 h_3v_t(L, 0), \alpha k w_t(L, 0), \left( \chi(L, 0), \eta(L, 0), \Theta(L, 0) \right)\right)\right)_{\mathbb{R}^3}\nonumber
\end{align}
where the inner product is in $\mathcal{H}_2$ and $W=\left( \chi, \chi_t, \eta, \eta_t, \Theta, \Theta_t,  \chi_t(L, t), \eta_t(L, t), \Theta_t(L, t) \right)^{\top} $ is solution of the adjoint system
\begin{equation}\label{sistema3}
\left\{\begin{array}{ll}
\rho_{1} h_1 \chi_{tt} - E_1 h_1 \chi_{xx} - k\left( -\chi + \eta + \alpha \Theta_x \right) = 0, & x \in (0,L),\;\; t\in (0, \tau) \\[1ex]
\rho_{3} h_3 \eta_{tt} - E_3 h_3 \eta_{xx} + k\left( -\chi + \eta + \alpha \Theta_x \right) = 0, & x \in (0,L),\;\; t\in (0, \tau) \\[1ex]
\rho h \Theta_{tt} + EI \Theta_{xxxx} - \alpha k \left( -\chi + \eta + \alpha \Theta_x \right)_x = 0, & x \in (0,L),\;\; t\in (0, \tau) \\[1ex]
\chi(0,t) = \eta(0,t) = \Theta_x(0,t) = \Theta_{xxx}(0,t) = 0, & t\in (0, \tau) \\[1ex]
\Theta_{xx}(L,t) = \Theta_{xxx}(L,t) = 0, & t\in (0, \tau) \\[1ex]
\chi_{tt}(L,t) + \chi_x(L,t) = 0, & t\in (0, \tau) \\[1ex]
\eta_{tt}(L,t) + \eta_x(L,t) = 0, & t\in (0, \tau) \\[1ex]
\Theta_{tt}(L,t) - \chi(L,t) + \eta(L,t) + \alpha \Theta_x(L,t) = 0, & t\in (0, \tau) \\[1ex]
\left( \chi, \eta, \Theta \right)(x,0) = \left( \chi_0, \eta_0, \Theta_0 \right)(x), & x \in (0,L) \\[1ex]
\left( \chi_t, \eta_t, \Theta_t \right)(x,0) = \left( \chi_1, \eta_1, \Theta_1 \right)(x), & x \in (0,L) \\
\left(  \chi, \eta, \Theta \right)  \left( L,\tau \right)  =(0, 0, 0).
\end{array}
\right.
\end{equation}
\end{definition}

The following proposition guarantees that the system \eqref{2bbm} admits a unique solution in the transposition sense.
\begin{proposition}
Let $T > 0$, $U_0 \in D\left(\mathcal{P}\right)  $, $(J_1, J_2, J_3) \in  \left(L^2\left(0,T;  \left( H_*^1(0,L) \right)^{*}\right)\right)^3 $
and $ f_i \in  L^2(0,T)$, for $i=1,2,3  $. Then, there exists a unique solution
$$ U = (u, u_t, v, v_t, w, w_t, u_t(L), v_t(L), w_t(L))^\top \in
C ([0,T]; D\left(\mathcal{P}\right)),$$of system \eqref{2bbm_nonhomo_2} which verifies \eqref{equa2}.
\end{proposition}
\begin{proof}
Let $T > 0$ and $\tau \in [0,T]$. According to Theorem \ref{ex-hom}, the system \eqref{sistema3} has a unique solution $W$ expressed as
\begin{align*}
W(t) = S(\tau - t) W^{\tau},
\end{align*}
and there exists a constant $C_T > 0$ such that
\begin{align}\label{deestiWparariesz}
\Vert W(t) \Vert_{\mathcal{H}_2} \leq C_T \Vert W^{\tau} \Vert_{\mathcal{H}_2}, \quad \forall\, t \in [0, \tau].
\end{align}

We define $\mathcal{R}$ as the linear functional corresponding to the right-hand side of \eqref{equa2}, namely
\begin{align*}
\mathcal{R}\left(W^{\tau} \right) &= \left\langle V_0^1, W(0)\right\rangle_{ \mathcal{H}_2}  +\displaystyle\int_{0}^{\tau} \, \left \langle \left(J_1(t), J_2(t), J_3(t)) \right), \left(\chi(t), \eta(t), \Theta(t) \right)\right\rangle_{\left[\left(  H_*^1(0,L) \right)^{*}, H_*^1(0,L) \right]^3}dt \nonumber\\
&+\left((E_1h_1f_1(t), E_3 h_3 f_2(t), \alpha k f_3(t)),  1_{(0, \tau)}(\chi(L, t), \eta(L, t), \Theta(L, t) )\right)_{[L^2(0, T)]^3} \nonumber \\
&+ \left( \left(E_1h_1u_t(L, 0), E_3 h_3v_t(L, 0), \alpha k w_t(L, 0), \left( \chi(L, 0), \eta(L, 0), \Theta(L, 0) \right)\right)\right)_{\mathbb{R}^3}
\end{align*}
where
\begin{align*}
&V_0^1= \bigg( \rho_1 h_1 u_1, - \rho_1 h_1 u_0, \rho_3 h_3 v_1,- \rho_3 h_3 v_0, \rho h w_1, \\ &\hspace{4cm} -\rho h w_0,
 - E_1 h_1 u(L, 0),
 -E_3 h_3 v(L, 0),
 -\alpha k w(L, 0)\bigg)^{\top}.
\end{align*}
From \eqref{deestiWparariesz} and the Cauchy–Schwarz inequality, we deduce
\begin{equation*}
\begin{split}
 |\mathcal{R}\left(W^{\tau}\right)| &\leq \Vert V^1_0 \Vert_{\mathcal{H}_2} \Vert W(0) \Vert_{\mathcal{H}_2} + C_T \Vert W(t) \Vert_{\mathcal{H}_2} \Vert(f_1, f_2, f_3) \Vert_{[L^2(0,T)]^3}\\
 &\quad + C_T \Vert W(t) \Vert_{C\left([0, \tau]; \mathcal{H}_2\right)} \Vert(J_1, J_2, J_3) \Vert_{ \left(L^2\left(0,T;  \left( H_*^1(0,L) \right)^{*}\right)\right)^3}\\
 &\quad + C_T \Vert W(t) \Vert_{C\left([0, \tau]; \mathcal{H}_2\right)} \Vert\left(u_t(L, 0), v_t(L, 0), w_t(L, 0)\right) \Vert_{ \mathbb{R}^3}\\
 & \leq C_T \bigg( \Vert V^1_0 \Vert_{\mathcal{H}_2} + \Vert(f_1, f_2, f_3) \Vert_{[L^2(0,T)]^3} + \Vert(J_1, J_2, J_3) \Vert_{ \left(L^2\left(0,T;  \left( H_*^1(0,L) \right)^{*}\right)\right)^3} \\
 &\hspace{4cm}+\Vert\left(u_t(L, 0), v_t(L, 0), w_t(L, 0)\right) \Vert_{ \mathbb{R}^3} \bigg) \Vert W^{\tau} \Vert_{\mathcal{H}_2}.
\end{split}
\end{equation*}
Therefore, we have $\mathcal{R} \in \mathcal{L}\big(D(\mathcal{P});\mathbb{R}\big)$. By the Riesz representation theorem, it follows that there exists a unique $U \in D(\mathcal{P})$ satisfying \eqref{equa2}. Moreover, the property $U \in C\big([0, T] ; D(\mathcal{P}) \big)$ follows by an argument analogous to that in \cite{CF2019}, and thus the details are omitted.
\end{proof}

\section{Controllability results} \label{sec5}
In this section, we study some boundary controllability properties of the system \eqref{2bbm}.
We start with the following characterization of a control driving system \eqref{2bbm} to zero. This kind of result is already classic for dispersive systems (see, for instance, \cite{BAPA} and \cite{M}).

\begin{lemma}\label{charac} The initial data  $U_0 \in D\left(\mathcal{P}\right)  $
is controllable to zero in time $T > 0$ with controls $ f_i \in  L^2(0,T),$ for $i=1,2,3  $, if and only if
\begin{align}\label{equiv}
 &-\displaystyle \Bigg\langle  \left(\begin{array}{c}
 \rho_1 h_1 u_1 \\  -\rho_1 h_1 u_0 \\  \rho_3 h_3 v_1\\ -\rho_3 h_3 v_0 \\
 \rho h  w_1 \\ -\rho h  w_0\\
 - E_1 h_1 u(L, 0)\\
 -  E_3 h_3 v(L, 0)\\
 -k \alpha w(L, 0)\end{array}\right), \left(\begin{array}{c}
 \chi(0)\\ \chi_t(0) \\ \eta(0) \\ \eta_t(0)\\ \Theta(0) \\ \Theta_t(0) \\\chi_t(L, 0)\\ \eta_t(L, 0)\\ \Theta_t(L, 0) \end{array}\right) \Bigg\rangle_{  \mathcal{H}_2} \\
&-\left( \left( E_1 h_1 w_t(L, 0),  E_3 h_3 v_t(L, 0),  k \alpha w_t(L, 0)\right); \left( \chi(L, 0), \eta(L, 0), \Theta(L, 0) \right)\right)_{\mathbb{R}^3} \nonumber \\
& =  E_1 h_1\int^T_0 f_1(t) \chi(L, t)dt +  E_3 h_3\int^T_0 f_2(t)\eta(L, t) dt+ k \alpha \int^T_0 f_3(t)\Theta(L, t) dt\nonumber
\end{align}
for any solution of the adjoint system
\begin{equation}\label{back}
\left\{\begin{array}{ll} \rho_{1} h_1\chi_{tt}- E_1 h_1\chi_{xx}-k\left(  -\chi+\eta + \alpha\Theta_x\right)     =0, &x\in(0,L),\,\,\,
t>0,\\
\rho_{3} h_3\eta_{tt}-E_3 h_3\eta_{xx}+k\left(   -\chi+\eta + \alpha\Theta_x\right)
=0,  &x\in(0,L),\,\,\,
t>0,\\
\rho h \Theta_{tt} -E I \Theta_{xxxx}-\alpha k \left(   -\chi+\eta + \alpha\Theta_x\right)_x
=0, & x\in(0,L),\,\,\, t>0,\\
\chi(0, t)=\eta(0, t)= \Theta_{x}(0, t)=\Theta_{xxx}(0, t) =0, & t>0,\\
\Theta_{xx}(L, t)=\Theta_{xxx}(L, t) =0, & t>0,\\
\chi_{tt}(L, t)+\chi_x(L,t)=0, &
t>0,\\
\eta_{tt}(L, t)+\eta_{x}(L, t)= 0, &
t>0,\\
\Theta_{tt}(L, t) -\chi_x(L,t)+\eta(L,t)+\alpha\Theta(L, t) =0, &
t>0,\\
\left(  \chi, \eta, \Theta \right)  \left( x,T \right)  =\left(\chi^{T}_0, \eta^T_0, \Theta_0^T\right)(x), & x\in(0,L),\\
\left( \chi_{t}, \eta_{t}, \Theta_{t}\right)  \left(  x,T\right)  =\left(\chi^{T}_1, \eta^T_1, \Theta_1^T\right)(x), & x\in(0,L),\end{array}\right.
\end{equation}
with $(\chi_0^{T}, \chi_1^{T}, \eta_0^{T}, \eta_1^{T}, \Theta_0^{T}, \Theta_1^{T}, \chi_1^{T}(L), \eta_1^{T}(L), \Theta_1^{T}(L))^{\top} \in D(\mathcal{P}^*)$. 
\end{lemma}
\begin{proof} The proof is a direct consequence of integration by parts and so we will omitted it.
\end{proof}

\begin{remark}\label{obs_des_obser}
It is well known that the variational equation \eqref{equiv} has a solution if there exists a positive
constant $C > 0$, such that the following observability inequality holds for any $$(\chi_0^{T}, \chi_1^{T}, \eta_0^{T}, \eta_1^{T}, \Theta_0^{T}, \Theta_1^{T}, \chi_1^{T}(L), \eta_1^{T}(L), \Theta_1^{T}(L))^{\top} \in D(\mathcal{P}^*)$$
\begin{equation}\label{DESIIMPORTANTE}
 \Vert W(0) \Vert^2_{\mathcal{H}_2} \leq C \int^T_0 \left( \left|\chi(L, t)\right|^2dt +  \left|\eta(L, t)\right|^2 + \left|\Theta(L, t)\right|^2 \right)dt,
\end{equation}where $W=\left( \chi, \chi_t, \eta, \eta_t, \Theta, \Theta_t,  \chi_t(L, t), \eta_t(L, t), \Theta_t(L,t) \right)^{\top} $ is solution of the adjoint system \eqref{back}.
\end{remark}

Now, we are in a position to prove the observability inequality \eqref{DESIIMPORTANTE}. We start with the following result.
\begin{proposition}
Let $U=(u,u_t, v,v_t, w,w_t, u_t(L,t), v_t(L,t), w_t(L,t))^{\top}$ the solution of the system \eqref{abs}. Then, for any $T>0$ and $s \in (\dfrac{1}{2}, 1)$, there exist a positive constant $C$ such that the following estimate hold
\begin{equation}\label{primerdesiobserder}
\begin{split}
\Vert U(t) \Vert^2_{\mathcal{H}_2}\leq C \int^T_0 \left( \left|u_t(L, t)\right|^2dt +  \left|v_t(L, t)\right|^2 + \left|w_t(L, t)\right|^2 \right)dt+C\Vert (u, v, w) \Vert^2_{L^{\infty}\left(0, T; [H^{s}(0, L)]^3\right)}.
\end{split}
\end{equation}
\end{proposition}
\begin{proof}
In order to obtain estimate \eqref{primerdesiobserder}, we multiply the first equation in \eqref{2bbm-hom} by $x u_x$ and integrate by parts on $(0, L)\times (0, T) $ to obtain
\begin{equation}\label{ecua1_1_fordesiguaob1}
\begin{split}
0 = & \rho_1 h_1 \int^L_0\left[ x u_t u_x \right]_0^Tdx -\dfrac{\rho_1 h_1 L}{2} \int^T_0 u^2_t(L, t)dt  + \dfrac{\rho_1 h_1}{2} \int^L_0 \int^T_0 u^2_tdt dx \\
& -\dfrac{E_1h_1L}{2} \int^T_0 u^2_x(L, t)dt  + \dfrac{E_1h_1}{2} \int^L_0 \int^T_0 u_x^2dt dx-k \int^L_0 \int^T_0 \left(   -u+v+\alpha w_{x}\right)x u_{x}dt dx.\end{split}
\end{equation}
Now, we multiply  the second equation in \eqref{2bbm-hom} by $x v_x$, the third one by $x w_x$, integrate by parts on $(0, L)\times (0, T) $ to get
\begin{equation}\label{ecua2_fordesiguaob1}
\begin{split}
 0 &= \int^L_0 \int^T_0 \left(   \rho_{3} h_3 v_{tt}-E_3 h_3 v_{xx}+k\left( -u+v+\alpha w_{x}\right)\right)x v_x dt dx \\
 & =  \rho_3 h_3 \int^L_0\left[ x v_t v_x \right]_0^Tdx -\dfrac{\rho_3 h_3 L}{2} \int^T_0 v^2_t(L, t)dt  + \dfrac{\rho_3 h_3}{2} \int^L_0 \int^T_0 v^2_tdt dx \\
& -\dfrac{E_3h_3 L}{2} \int^T_0 v^2_x(L, t)dt  + \dfrac{E_3 h_3}{2} \int^L_0 \int^T_0 v_x^2dt dx -k \int^L_0 \int^T_0 \left(   -u+v+\alpha w_{x}\right)x v_{x}dt dx,
\end{split}
\end{equation}
and
\begin{equation}\label{ecua3_fordesiguaob1}
\begin{split}
 0 &= \int^L_0 \int^T_0 \left(  \rho h w_{tt}+EI w_{xxxx}-\alpha k \left( -u+v+\alpha w_{x}\right)_x  \right)x w_x dt dx \\
 & =  \rho h \int^L_0\left[ x w_t w_x \right]_0^Tdx -\dfrac{ \rho h L}{2}  \int^T_0 w^2_{t}(L, t)dt + \dfrac{\rho h }{2}  \int^L_0\int^T_0 w^2_{t}dt dx\\
 &+ \dfrac{3 E I }{2}  \int^L_0\int^T_0 w^2_{xx}dt dx -k \alpha\int^T_0 L \left(  -u(L, t)+v(L, t)+\alpha w_x(L, t)\right)w_{x}(L, t)dt \\
 & +k \alpha \int^L_0 \int^T_0 \left(   -u+v+\alpha w_{x}\right)x w_{xx}dt dx + k \int^L_0 \int^T_0\left(   -u+v+\alpha w_{x}\right)^2 dt dx  \\
 & -k \int^L_0 \int^T_0\left(  -u+v+\alpha w_{x}\right)\left( -u+v\right) dt dx,
\end{split}
\end{equation}
respectively. A  straightforward computation shows that
\begin{equation}\label{I_0pri}
\begin{split}
 &k \int^L_0 \int^T_0 \left(  -u+v+\alpha w_{x}\right)\left(  -u+v+\alpha w_{x}\right)_x x dt dx \\
 &= \dfrac{kL}{2} \int^T_0 \left(  -u(L, t)+v(L, t)+\alpha w_x(L, t)\right)^2dt -\dfrac{k}{2}\int^L_0 \int^T_0 \left(  -u+v+\alpha w_{x}\right)^2 dt dx.
\end{split}
\end{equation}
Thus, adding the identities \eqref{ecua1_1_fordesiguaob1}-\eqref{ecua3_fordesiguaob1} and by taking into account \eqref{I_0pri}, it follows that
\begin{equation}\label{total_suma_some}
\begin{split}
&\dfrac{1}{2} \int^T_0 \Vert U(t) \Vert^2_{\mathcal{H}_2}dt \\
&= \dfrac{1}{2}\int^T_0 \left( ( \rho_1 h_1L+E_1 h_1)u^2_{t}(L, t) +(( \rho_3 h_3L+E_3 h_3)v^2_{t}(L, t)+ (\rho h L+\alpha k)w^2_{t}(L, t)\right)dt \\
&- \int^L_0\left[\rho_1 h_1 x u_t u_x + \rho_3 h_3 x v_t v_x +  \rho h  x w_t w_x \right]_0^Tdx + \dfrac{  L}{2}  \int^T_0 \left( E_1 h_1u^2_{x}(L, t)+E_3 h_3v^2_{x}(L, t)\right)dt  \\
& + k \int^L_0 \int^T_0\left(  -u+v+\alpha w_{x}\right)\left( -u+v\right) dt dx\\&+\dfrac{\alpha k L}{2}\int^T_0  \left(  -u(L, t)+v(L, t)+\alpha w_x(L, t)\right)w_{x}(L, t)dt  \\
& -\dfrac{kL}{2}\int^T_0  \left(  -u(L, t)+v(L, t)+\alpha w_x(L, t)\right)\left( -u(L, t)+v(L, t)\right)dt.
\end{split}
\end{equation}
Applying the Young's inequality, for any  $s \geq 0$ and  $\varepsilon>0$,   we have that
\begin{equation}\label{EQ1P}
\begin{split}
 &3k \int^L_0 \left(  -u+v+\alpha w_{x}\right)\left( -u+v\right)  dx \\
 &\leq  \varepsilon \left(  k\left\|( -u+v+\alpha w_{x})(\cdot, t)\right\|^2 + E_1 h_1\left\|u_x(\cdot, t) \right\|^2 + E_3 h_3\left\|v_x(\cdot, t) \right\|^2\right)\\
 &+ C_{\varepsilon,k} \left\|(u, v)(\cdot, t) \right\|_{[H^s(0, L)]^2}^2.
 \end{split}
\end{equation}

On the other hand, for $s \in (\dfrac{1}{2}, 1)$,   using Young's inequality and  from the Sobolev embedding (see, for instance, \cite{DPV,LiMa}), we obtain a  positive constant $C_{\varepsilon_1,k, E, L}$, such that
\begin{equation}\label{EQ2P}
\begin{split}
& \dfrac{L}{2}\left \vert\left(  E_1 h_1 u^2_{x}(L, t)+E_3 h_3 v^2_{x}(L, t) + \alpha k   \left(  -u(L, t)+v(L, t)+\alpha w_x(L, t)\right)w_{x}(L, t) \right. \right. \\
&\left.\left.-k   \left(  -u(L, t)+v(L, t)+\alpha w_x(L, t)\right)\left( -u(L, t)+v(L, t)\right) \right. \right \vert \\
& \leq  \varepsilon_1 \left(  k\left\|( -u+v+\alpha w_{x})(\cdot, t)\right\|^2 + E_1 h_1\left\|u_x(\cdot, t) \right\|^2 + E_3 h_3\left\|v_x(\cdot, t) \right\|^2\right)  \\
&+ C_{\varepsilon_1,k, E, L}\left\| (u, v, w)(\cdot, t) \right\|_{[H^s(0, L)]^3}^2,
\end{split}
\end{equation}
for any $\varepsilon_1>0$. Now, for all $t \in [0, T]$ and applying once more Young's inequality,  we get that
\begin{equation}\label{EQ3P}
\begin{split}
&\left \vert\int^L_0\left(\rho_1 h_1 x u_t u_x + \rho_3 h_3 x v_t v_x +  \rho h  x w_t w_x \right)(x,t)dx\right\vert \\
& \leq \varepsilon_2 \left( \rho_1 h_1 \Vert u_t(\cdot, t) \Vert^2 + \rho_3 h_3 \Vert v_t(\cdot, t) \Vert^2 + \rho h \Vert w_t(\cdot, t) \Vert^2 \right) + C_{\varepsilon_2,\rho, L}\left\| (u, v, w)(\cdot, t) \right\|_{[H^s(0, L)]^3}^2,
\end{split}
\end{equation}
for any $\varepsilon_2>0$, $s \in (\dfrac{1}{2}, 1)$ and some constant $C_{\varepsilon_2,\rho_1, \rho_2, L}>0$.
By combining \eqref{total_suma_some}-\eqref{EQ3P} and taking $\varepsilon=\varepsilon_1=\varepsilon_2$, it follows that
\begin{align*}
\dfrac{1}{2} \int^T_0 \Vert U(t) \Vert^2_{\mathcal{H}_2}dt \leq& 2\varepsilon \int^T_0 \Vert U(t) \Vert^2_{\mathcal{H}_2}dt + C\int^T_0 \left( u^2_{t}(L, t) +v^2_{t}(L, t)+ w^2_{t}(L, t)\right)dt\\
& +C_{\varepsilon}\int^T_0 \left\| (u, v, w)(\cdot, t) \right\|_{[H^s(0, L)]^3}^2dt.
\end{align*}
From estimative above and we choose $\varepsilon$ small enough  so that $0<\varepsilon<\dfrac{1}{4}$, we immediately deduce
\begin{equation}\label{prefinalesti_impor0}
\begin{split}
\int^T_0 \Vert U(t) \Vert^2_{\mathcal{H}_2}dt \leq &C\int^T_0 \left( u^2_{t}(L, t) +v^2_{t}(L, t)+ w^2_{t}(L, t)\right)dt+C\Vert (u, v, w) \Vert^2_{L^{\infty}\left(0, T; [H^{s}(0, L)]^3\right)},
\end{split}
\end{equation}
 for any  $s \in (\dfrac{1}{2}, 1).$ Finally, using \eqref{identi_gisometr} and the inequality \eqref{prefinalesti_impor0}, we infer the estimative \eqref{ecua1_1_fordesiguaob1}.  This finishes the proof.
\end{proof}

With the previous results in hand, we can establish an intermediary observability inequality as follows.
\begin{proposition}\label{teoimportantepaperI}
Let $U=(u,u_t, v,v_t, w,w_t, u_t(L,t), v_t(L,t), w_t(L,t))^{\top}$ the solution of the system \eqref{abs}. Then,  there exists a positive constant $C$ such that the following estimate holds
\begin{align}\label{seconddesiobserder2}
\Vert U_0 \Vert^2_{\mathcal{H}_2} \leq C \int^T_0 \left( \left|u_t(L, t)\right|^2dt +  \left|v_t(L, t)\right|^2 + \left|w_t(L, t)\right|^2 \right)dt.
\end{align}
\end{proposition}
\begin{proof}
Let us argue by contradiction, following the so-called \emph{compactness-uniqueness} argument (see, for instance, \cite{ZUA}).
Suppose that \eqref{seconddesiobserder2} does not hold.
Hence, there exists a sequence $U^n_0 \in \mathcal{H}_2$ such that
\begin{align}\label{noriguaconver_1}
\Vert U^n(t) \Vert^2_{\mathcal{H}_2} = \Vert U^n_0 \Vert^2_{\mathcal{H}_2} = 1,
\quad \text{for all } t \in [0,T],
\end{align}
and
\begin{align}\label{dericonver_2}
\int_0^T \!\left(
\left|u^n_t(L, t)\right|^2 +
\left|v^n_t(L, t)\right|^2 +
\left|w^n_t(L, t)\right|^2
\right) dt
\longrightarrow 0,
\quad \text{as } n \rightarrow \infty.
\end{align}
Here, for each $n \in \mathbb{N}$, the function
\[
U^n = \big(u^n, u^n_t, v^n, v^n_t, w^n, w^n_t, u^n_t(L,t), v^n_t(L,t), w^n_t(L,t)\big)^{\top}
\]
solves the following system:
\begin{equation}\label{abswithn}
\left\{
\begin{array}{l}
U^n_t(t) = \mathcal{P} U^n(t), \\[0.3em]
U^n(0) = U^n_0.
\end{array}
\right.
\end{equation}
From \eqref{noriguaconver_1} and the well-posedness results, we obtain the following uniform bounds:
\begin{equation}\label{paracompac_1}
\left|
\begin{array}{lll}
(u^n, v^n, w^n)  & \text{bounded in} & L^{\infty}\!\left(0, T; [H^{1}(0, L)]^3\right), \\[0.3em]
(u^n_t, v^n_t, w^n_t) & \text{bounded in} & L^{\infty}\!\left(0, T; [L^{2}(0, L)]^3\right).
\end{array}
\right.
\end{equation}

Once we have the embedding
\[
H^{1}(0, L)\hookrightarrow H^s(0,L) \hookrightarrow L^{2}(0, L),
\]
and since the embedding $H^{1}(0, L)\hookrightarrow H^s(0,L)$ is compact for all $s \in \left(\tfrac{1}{2}, 1\right)$, from \eqref{paracompac_1} and the classical compactness results given by \cite[Corollary~4]{SIMO}, we can extract a subsequence of $(u^n, v^n, w^n)$, still denoted by $(u^n, v^n, w^n)$, such that
\begin{align}\label{firsconverimp}
(u^n, v^n, w^n) \longrightarrow (u, v, w) \quad \text{strongly in} \quad L^{\infty}\!\left(0, T; [H^{s}(0, L)]^3\right).
\end{align}
Thus, from \eqref{primerdesiobserder}, \eqref{dericonver_2}, and \eqref{firsconverimp}, we deduce that $(U^n)$ is a Cauchy sequence in $L^{\infty}\!\left(0, T; \mathcal{H}_2\right)$. Therefore,
\begin{align*}
U^n \longrightarrow U \quad \text{strongly in} \quad L^{\infty}\!\left(0, T; \mathcal{H}_2\right),
\end{align*}
and by \eqref{noriguaconver_1},
\begin{align}\label{nor1soluU}
\Vert U(t) \Vert^2_{\mathcal{H}_2} = 1, \quad \text{for all } t \in [0,T].
\end{align}

Also,
\begin{align}\label{cerofronteraU}
0 &= \liminf_{n \rightarrow 0} \left\{ \int^T_0 \!\!\left(
\left|u^n_t(L, t)\right|^2 +
\left|v^n_t(L, t)\right|^2 +
\left|w^n_t(L, t)\right|^2
\right) dt \right\} \\
&\geq \int^T_0 \!\!\left(
\left|u_t(L, t)\right|^2 +
\left|v_t(L, t)\right|^2 +
\left|w_t(L, t)\right|^2
\right) dt. \nonumber
\end{align}
From \eqref{cerofronteraU}, we have that the limit
\[
U = \big(u, u_t, v, v_t, w, w_t, u_t(L,t), v_t(L,t), w_t(L,t)\big)^{\top}
\]
satisfies the linear problem
\begin{equation}\label{2bbm_nonhomo_Ulimi}
\left\{
\begin{array}{ll}
\rho_{1 } h_1 u_{tt}-E_1h_1u_{xx}-k\!\left(  -u+v+\alpha w_{x}\right)   =0, & x\in(0,L),\, t\in (0, T),\\[0.5em]
\rho_{3} h_3 v_{tt}-E_3 h_3 v_{xx}+k\!\left( -u+v+\alpha w_{x}\right)   =0,  & x\in(0,L),\, t \in (0, T),\\[0.5em]
\rho h w_{tt}+EI w_{xxxx}-\alpha k \!\left( -u+v+\alpha w_{x}\right)_x  =0, & x\in(0,L),\, t\in (0, T),\\[0.5em]
u(0,t)=v(0,t)=w_x(0, t)=w_{xxx}(0,t)=0, & t\in (0, T),\\[0.5em]
w_{xx}(L,t)=w_{xxx}(L,t)=0, & t\in (0, T),\\[0.5em]
u_{x}(L, t)=0, \quad v_x(L, t)=0, & t\in (0, T),\\[0.5em]
-u(L,t)+v(L, t)+\alpha w_x(L,t)=0, & t\in (0, T).
\end{array}
\right.
\end{equation}

Let $( \hat{u}, \hat{v}, \hat{w})$ denote the extension by zero of $(u_{tt}, v_{tt}, w_{tt})$ for $x \in (-a,a)\setminus(0,L)$, where $(-a,a)\supset (0,L)$ is an interval. Then, $( \hat{u}, \hat{v}, \hat{w})$ solves, in $\mathcal{D}'\big((0, L)\times (0, T)\big)$, the following system:
\begin{equation}\label{bbm_nonhomo_extensionessolu}
\left\{\begin{array}{ll}
\rho_{1 } h_1 \hat{u}_{tt}-E_1h_1\hat{u}_{xx}-k\left(  -\hat{u}+\hat{v}+\alpha \hat{w}_{x}\right)   =0, &x\in(0,L),\,\,\,
t\in (0, T)\\
\rho_{3} h_3 \hat{v}_{tt}-E_3 h_3 \hat{v}_{xx}+k\left( -\hat{u}+\hat{v}+\alpha \hat{w}_{x}\right)
=0,  &x\in(0,L),\,\,\,
t \in (0, T)\\
\rho h \hat{w}_{tt}+EI \hat{w}_{xxxx}-\alpha k \left( -\hat{u}+\hat{v}+\alpha \hat{w}_{x}\right)_x
=0, & x\in(0,L),\,\,\, t\in (0, T)\\
\hat{u}(0,t)=\hat{v}(0,t)=\hat{w}_x(0, t)=\hat{w}_{xxx}(0,t) =0, & t\in (0, T)\\
\hat{w}_{xx}(L,t)=\hat{w}_{xxx}(L,t) =0, & t\in (0, T)\\
\hat{u}_{x}(L, t)= 0, &
t\in (0, T)\\
 \hat{v}_x(L, t) =0, &
t\in (0, T)\\
-\hat{u}(L,t)+\hat{v}(L, t)+\alpha \hat{w}_x(L,t)=0, &
t\in (0, T)\\
( \hat{u}, \hat{v}, \hat{w}) =(0, 0, 0), & x \in (-a, 0] \cup [L, a), \ t\in (0, T),
\end{array}\right.
\end{equation} Then, by Holmgren’s Uniqueness Theorem \cite{Holmgren}, $( \hat{u}, \hat{v}, \hat{w}) \equiv (0, 0, 0) $ in $(-a, a)\times(0,T)$. Hence, from \eqref{2bbm_nonhomo_Ulimi}, we have that  $U=(u,u_t, v,v_t, w,w_t, u_t(L,t), v_t(L,t), w_t(L,t))^{\top}$ satisfies
\begin{equation}\label{2bbm_nonhomo_Ulimi_deridoscero}
\left\{\begin{array}{ll}
-E_1h_1u_{xx}-k\left(  -u+v+\alpha w_{x}\right)   =0, &x\in(0,L),\,\,\,
t\in (0, T)\\
-E_3 h_3 v_{xx}+k\left( -u+v+\alpha w_{x}\right)
=0,  &x\in(0,L),\,\,\,
t \in (0, T)\\
EI w_{xxxx}-\alpha k \left( -u+v+\alpha w_{x}\right)_x
=0, & x\in(0,L),\,\,\, t\in (0, T)\\
u(0,t)=v(0,t)=w_x(0, t)=w_{xxx}(0,t) =0, & t\in (0, T)\\
w_{xx}(L,t)=w_{xxx}(L,t) =0, & t\in (0, T)\\
u_{x}(L, t)= 0, &
t\in (0, T)\\
 v_x(L, t) =0, &
t\in (0, T)\\
-u(L,t)+v(L, t)+\alpha w_x(L,t)=0, &
t\in (0, T).\end{array}\right.
\end{equation}We infer, from the proof of Lemma \ref{0inresolv}, the unique solution of \eqref{2bbm_nonhomo_Ulimi_deridoscero} is given by $U=0$, and this
 contradicts \eqref{nor1soluU}. Then, necessarily, \eqref{seconddesiobserder2} holds. This completes the proof of the Theorem.
\end{proof}

\begin{remark}\label{directinequali}
From \eqref{identi_gisometr}, we directly verify that
\begin{align}\label{direcdirec}
\int_0^T \!\left(
\left|u_t(L, t)\right|^2 +
\left|v_t(L, t)\right|^2 +
\left|w_t(L, t)\right|^2
\right) dt
\leq C \Vert U_0 \Vert^2_{\mathcal{H}_2}.
\end{align}
\end{remark}

We are now in a position to prove the second main result of this section, namely, the observability inequality that provides a positive answer to the control problem $P_3$.

\begin{theorem}\label{teoprindesiusados}
For any $U_0 \in D(\mathcal{P})$, let
\[
U = (u, u_t, v, v_t, w, w_t, u_t(L,t), v_t(L,t), w_t(L,t))^{\top}
\]
be the solution of the system \eqref{abs}. Then, there exists a positive constant $C$ such that the following estimate holds:
\begin{align}\label{thirdesiobserder3}
\Vert U_0 \Vert^2_{\mathcal{H}_2}
\leq
C \int_0^T \!\left(
\left|u(L, t)\right|^2 +
\left|v(L, t)\right|^2 +
\left|w(L, t)\right|^2
\right) dt.
\end{align}
\end{theorem}

\begin{proof}
In order to prove inequality \eqref{thirdesiobserder3}, we consider the auxiliary problem
\begin{equation}\label{abssombrero}
\left\{
\begin{array}{l}
\tilde{U}_t = \mathcal{P} \tilde{U}, \\[0.3em]
\tilde{U}(0) = \tilde{U}_0,
\end{array}
\right.
\end{equation}
where $\tilde{U}_0 = \mathcal{P}^{-1} U_0$ (the existence of $\mathcal{P}^{-1}$ is guaranteed by Lemma \ref{0inresolv}). Thus, from Proposition \ref{teoimportantepaperI}, the solution
\[
\tilde{U} = (\tilde{u}, \tilde{u}_t, \tilde{v}, \tilde{v}_t, \tilde{w}, \tilde{w}_t, \tilde{u}_t(L,t), \tilde{v}_t(L,t), \tilde{w}_t(L,t))^{\top}
\]
of system \eqref{abssombrero} satisfies
\begin{align}\label{tresdesiobserderprima}
\Vert \tilde{U}_0 \Vert^2_{\mathcal{H}_2}
\leq
C \int_0^T \!\left(
\left|\tilde{u}_t(L, t)\right|^2 +
\left|\tilde{v}_t(L, t)\right|^2 +
\left|\tilde{w}_t(L, t)\right|^2
\right) dt,
\end{align}
for some constant $C > 0$.

On the other hand, it is straightforward to see that $\tilde{U}_t = U$. Hence,
\begin{align}\label{equalderiandwxis}
(\tilde{u}_t, \tilde{v}_t, \tilde{w}_t) = (u, v, w).
\end{align}

Finally, combining \eqref{tresdesiobserderprima}, \eqref{equalderiandwxis}, and applying the Riesz representation theorem, we deduce that
\begin{align*}
\Vert U_0 \Vert^2_{\mathcal{H}_2}
&\leq C \Vert \tilde{U}_0 \Vert^2_{\mathcal{H}_2} \\
&\leq C \int_0^T \!\left(
\left|\tilde{u}_t(L, t)\right|^2 +
\left|\tilde{v}_t(L, t)\right|^2 +
\left|\tilde{w}_t(L, t)\right|^2
\right) dt \\
&= C \int_0^T \!\left(
\left|u(L, t)\right|^2 +
\left|v(L, t)\right|^2 +
\left|w(L, t)\right|^2
\right) dt,
\end{align*}
and the proof is complete.
\end{proof}

A direct consequence of the previous theorem is the following.

\begin{corollary}\label{finaldesforadforcontrolteo}
Let
\[
W = (\chi, \chi_t, \Theta, \Theta_t, \eta, \eta_t, \chi_t(L, t), \eta_t(L, t), \Theta_t(L, t))^{\top}
\]
be the solution of the adjoint system \eqref{back}. Then, there exists a constant $C > 0$ such that the following observability inequality holds:
\begin{equation}\label{DESIIMPORTANTE_colorarioprobado}
\Vert W(0) \Vert^2_{\mathcal{H}_2}
\leq
C \int_0^T \!\left(
\left|\chi(L, t)\right|^2 +
\left|\eta(L, t)\right|^2 +
\left|\Theta(L, t)\right|^2
\right) dt.
\end{equation}
\end{corollary}

The observability inequality established in Corollary \ref{finaldesforadforcontrolteo}, combined with the HUM method from control theory, allows us to obtain the second main outcome of this paper, namely, the controllability of the system stated in Theorem \ref{controlteorfinalnodemos}.

\section{Conclusion}
This paper considered two different problems related to a linear Rao-Nakra type sandwich beam, which consists of three coupled PDEs. An exponential stabilization result is obtained for the first problem, where the boundary conditions are static and subject to three time-dependent boundary delayed controls, while three time-varying interior damping controls are present. It would be desirable to investigate, in a future work, the stabilization problem with the presence of at most two interior damping controls. Furthermore, a promising research avenue is to tackle the problem where the boundary controls are of memory-type.

For the second problem, the boundary conditions are supposed to be dynamical. Then, a null controllability result is shown via the action of three boundary controls. Another interesting question is whether one can achieve the controllability result by means of at most two boundary controls.

\subsection*{Acknowledgements} The authors sincerely thank the anonymous reviewers for their careful reading and constructive comments, which have significantly improved the manuscript.

\subsection*{Funding.}G. Bautista was supported by FAPERJ under the program PDS-2024, grant number SEI-0260003/019497/2024. Capistrano–Filho was partially supported by CAPES-COFECUB program grant number 88887.879175/2023-00, CNPq grant numbers 301744/2025-4 and 421573/2023-, and PROPG (UFPE) \textit{via} PROAP resources. L\'imaco was partially supported by CNPq grant number 310860/2023-7 (Brazil). O. Sierra was supported by FAPERJ (Rio de Janeiro, Brazil) \textit{via} process  SEI 260003/000175/2024 and by Escola de Matemática Aplicada, Fundação Getúlio Vargas (Rio de Janeiro, Brazil).

\subsection*{Data Availability} No data were created or analyzed in this study.

\subsection*{Conflict of interest} The authors declare that they have no conflict of interest.

\end{document}